\documentclass[a4paper,UKenglish,cleveref, autoref, thm-restate]{article}

\usepackage{amsmath,mfl,amssymb,stmaryrd,amsthm,amsfonts,MnSymbol}

\newcommand{\wt}{\mathrm{wt}}

\usepackage{times, color}
\usepackage{algorithmic}
\usepackage{graphicx}
\usepackage{textcomp, hyperref}
\usepackage{mathtools}

\DeclareMathAlphabet{\mathbbold}{U}{bbold}{m}{n}

\usepackage[numbers]{natbib}

\usepackage{xifthen}
\usepackage{dsfont,amsmath,amssymb,amsthm,mathtools,stmaryrd,bbm,color}

\DeclarePairedDelimiter{\ser}{\llbracket}{\rrbracket}

\allowdisplaybreaks
\numberwithin{equation}{section}

\newcommand{\defas}{\vcentcolon=}
\newcommand{\ddefas}{\vcentcolon\vcentcolon=}

\newcommand{\edge}{\mathrm{edge}}
\newcommand{\cut}{\mathrm{cut}}
\newcommand{\clique}{\mathrm{clique}}

\newcommand{\free}{\mathrm{Free}}

\newcommand{\true}{{\textup{\texttt{true}}}}
\newcommand{\false}{{\textup{\texttt{false}}}}
\newcommand{\trop}{\mathrm{Trop}}
\newcommand{\arct}{\mathrm{Arct}}

\newcommand{\Sr}{S}

\newcommand{\Nz}{\mathbb{N}}
\newcommand{\Np}{\mathbb{N}_+}

\newcommand{\splus}{+}
\newcommand{\smal}{\cdot}
\newcommand{\bigsplus}{\sum}
\newcommand{\bigsmal}{\prod}
\newcommand{\szero}{\mathbbold{0}}
\newcommand{\sone}{\mathbbold{1}}
\newcommand{\bzero}{0}
\newcommand{\bone}{1}
\newcommand{\fplus}{\oplus}
\newcommand{\fmal}{\otimes}
\newcommand{\bigfplusop}{{\textstyle \bigoplus}} 
\newcommand{\bigfplus}{\bigoplus} 
\newcommand{\bigfmalop}{{\textstyle \bigotimes}} 
\newcommand{\bigfmal}{\bigotimes} 

\newcommand{\rel}[1]{\mathrm{Rel}_{#1}}
\newcommand{\ar}[1]{\mathrm{ar}_#1}

\newcommand{\inter}[1]{\mathcal{I}_{#1}}
\newcommand{\str}[1]{\mathrm{Str}(#1)}

\newcommand{\fo}[1]{\mathrm{FO}(#1)}
\newcommand{\wfo}[1]{\mathrm{wFO}(#1)}

\newcommand{\mso}[1]{\mathrm{SO}(#1)}
\newcommand{\wmso}[1]{\mathrm{wSO}(#1)}

\newcommand{\dom}{\mathrm{dom}}

\newcommand{\B}{\mathbb{B}}

\newcommand{\Q}{\mathbb{Q}}
\newcommand{\R}{\mathbb{R}}

\newcommand{\mcp}{\mathcal{P}}

\newcommand{\mcv}{\mathcal{V}}

\newcommand{\mcx}{\mathcal{X}}

\newcommand{\mfa}{\mathfrak{A}}

\newcommand{\mfg}{\mathfrak{G}}

\newcommand{\vecg}[1]{\bar{#1}} 

\newtheorem{Def}{Definition}
\newtheorem{Thm}{Theorem}
\newtheorem{Rmk}[Thm]{Remark}
\newtheorem{Convention}[Thm]{Convention}
\newtheorem{Exm}[Thm]{Example}
\newtheorem{Cor}[Thm]{Corollary}

\newtheorem{Pro}[Thm]{Proposition}

\title{Logical Characterizations of Weighted Complexity Classes} 

\begin{document}

\author{Guillermo Badia \\ University of Queensland \\ \\ Manfred Droste \\ Leipzig University \\ \\ Carles Noguera \\ University of Siena \\ \\ Erik Paul \\ Leipzig University}

\maketitle

\begin{abstract}
Fagin's seminal result characterizing  $\mathsf{NP}$ in terms of existential second-order logic started  the fruitful field of descriptive complexity theory.
In recent years, there has been much interest in the investigation of quantitative (weighted) models of computations. In this  paper, we start the study of descriptive complexity based on weighted Turing machines over arbitrary semirings. We provide machine-independent characterizations (over ordered structures) of the weighted complexity classes $\mathsf{NP}[\mathcal{S}], \mathsf{FP}[\mathcal{S}]$, $\mathsf{FPLOG}[\mathcal{S}]$, $\mathsf{FPSPACE}[\mathcal{S}]$, and $\mathsf{FPSPACE}_{poly}[\mathcal{S}]$  in terms of definability in suitable weighted logics for an arbitrary semiring $\mathcal{S}$. In particular, we prove weighted versions of Fagin's theorem (even for arbitrary structures, not necessarily ordered, provided that the semiring is idempotent and commutative), the Immerman--Vardi's theorem (originally for $\mathsf{P}$)  and the Abiteboul--Vianu--Vardi's theorem (originally for $\mathsf{PSPACE}$). We also address a recent open problem proposed by Eiter and Kiesel.

 Recently, the above mentioned weighted complexity classes  have been investigated  in  connection to classical counting complexity classes. Furthermore, several classical counting complexity classes
have been characterized in terms of particular weighted logics over the semiring $\mathbb{N}$  of natural numbers. In this work, we cover several of these classes and obtain new results for others such as  $\mathsf{NPMV}$, $\oplus \mathsf{P}$, or  the collection of real-valued languages realized
by polynomial-time real-valued nondeterministic Turing machines.  Furthermore, our results apply to classes based on many other important semirings, such as the max-plus and the min-plus semirings over the natural numbers which correspond to the classical classes $\mathsf{MaxP}[O(\log n)]$ and $\mathsf{MinP}[O(\log n)]$, respectively.

\end{abstract}

\section{Introduction}

Descriptive complexity is a branch of computational complexity, as well as finite model theory, where the difficulty in solving a problem by a Turing machine is characterized not by the amount of resources required (such as time, space and so on) but rather in terms of the complexity of describing the problem in some logical formalism. This field was initially started in 1974 by Ronald Fagin with the celebrated result in~\cite{Fagin} (coined by Neil Immerman as `Fagin's theorem') which stated that the class of $\mathsf{NP}$
languages coincides with the class of languages definable in existential second-order logic. Many further surprising results followed this development, particularly the Immerman--Vardi's theorem characterizing  $\mathsf{P}$ over ordered structures using fixed-point logic~\cite{Immerman1, Vardi}  and the Abiteboul--Vianu--Vardi characterization of $\mathsf{PSPACE}$ in terms of partial fixed-point logic~\cite{Vianu,Vardi}. Today there are several textbooks that cover the fundamentals of the area as a line of research within finite model theory~\cite{E-F,Libkin,G-K-L,Immerman2}. In this paper, we propose to study quantitative versions of some of these key results in this important field in connection with weighted computation. We work over finite structures that come with a linear ordering, which is a standard restriction in descriptive complexity.

Weighted automata are nondeterministic finite automata augmented with values from a semiring as weights on the transitions~\cite{Schutzenberger:1961}. These weights may model, e.g.\ the cost involved when executing a transition, the amount of resources or time needed for this, or the probability or reliability of its successful execution.
The theory of weighted automata and weighted context-free grammars was essential for the solution of such classical automata-theoretic problems as the decidability of the equivalence of unambiguous context-free languages and regular languages~\cite{Salomaa:78} (in fact, the only known proofs of this involve weighted automata), the decidability of two given deterministic multitape automata~\cite{Harju}, and the decidability of two given  deterministic pushdown automata~\cite{Meitus,Senizergues}. This led to quick development of this field, described in the books~\cite{Berstel-Perrin-Reutenauer,Droste:Handbook,Eilenberg,Kuich-Salomaa,Sakarovitch,Salomaa:78}.
Furthermore, weighted automata and weighted context-free grammars have been used as basic concepts in natural language processing and speech recognition, as well as in algorithms for digital image compression~\cite{Albert-Kari}.
Weighted logic~\cite{DrosteGastin:07}, with weights in an arbitrary semiring, was developed originally to obtain a weighted version of the B\"uchi--Elgot--Trakhtenbrot theorem, showing that a certain weighted monadic second-order logic has the same expressive power on words as weighted automata. Consequently, this weighted logic over suitable semirings like fields has similar decidability properties
on words as unweighted monadic second-order logic. It is worth remarking  that the classical B\"uchi--Elgot--Trakhtenbrot theorem is usually regarded as part of the ``prehistory'' of descriptive complexity~\cite[p.~145]{G-K-L}.

Weighted Turing machines extend the concept of weighted automata as natural quantitative counterparts of classical Turing machines. They were first introduced under the name `algebraic Turing machines' in~\cite{Damm, Damm1} and they have attracted further attention in~\cite{Kos}. Instances of this concept include the so called  `fuzzy Turing machines'~\cite{Wie, Be}. Recently, 
the articles~\cite{Eiter, Eiter1} have introduced a related notion of `semiring Turing machine' and explicitly asked for the development of descriptive complexity in such framework as an open problem, focusing specifically on Fagin's theorem in connection to weighted logic~\cite[p.\ 255]{Eiter}. We will address this problem at the end of Section~\ref{complexitysec}.

\medskip

\noindent
{\bf Our contribution.} The present paper develops a theory of weighted descriptive complexity and establishes quantitative versions of some celebrated classical theorems. The novel contributions of this work can be summarized in the following characterizations (for an arbitrary semiring $\mathcal{S}$):
\begin{itemize}
\item The weighted complexity class $\mathsf{NP}[\mathcal{S}]$ coincides with the queries definable by weighted existential second-order logic on ordered structures, with weights in $\mathcal{S}$, respectively for all structures if $\mathcal{S}$ is idempotent and commutative (Theorem~\ref{thm:wfagin}).
\item The weighted complexity class $\mathsf{FP}[\mathcal{S}]$ coincides with the queries definable by weighted inflationary fixed-point logic, with weights in $\mathcal{S}$ (Theorem~\ref{I-V}).
\item The weighted complexity class $\mathsf{FPSPACE}[\mathcal{S}]$ coincides with the queries definable by weighted partial fixed-point logic with the addition of second-order multiplicative and additive quantifiers, with weights in $\mathcal{S}$ (Theorem~\ref{space}).
\item The weighted complexity class $\mathsf{FPSPACE}_{poly}[\mathcal{S}]$ coincides with the queries definable by weighted partial fixed-point logic, with weights in $\mathcal{S}$ (Theorem~\ref{spacepoly}).
\item The weighted complexity class $\mathsf{FPLOG}[\mathcal{S}]$ coincides with the queries definable by weighted deterministic transitive closure logic (Theorem~\ref{logi}).
\end{itemize}
\medskip

\noindent
{\bf Related work.} We should remark that  the article~\cite{Arenas} (following up on the work of~\cite{Saluja:95}) proposes the idea of using certain weighted logics (with weights in the semiring $\mathbb{N}$  of natural numbers or, in a couple of cases, $\mathbb{Z}$) to characterize well-known counting complexity classes. The authors obtain several interesting results that are also covered by our more encompassing work here (that is, they provide logical characterizations of  $\#\mathsf{P}$, $ \mathsf{FP}$, $\mathsf{FPSPACE}$, $\mathsf{FPSPACE}(poly)$, $\mathsf{GapP}$, and $\mathsf{MaxP}$).  There is, however, some orthogonality as they cover some classical complexity classes that we do not and, similarly, we cover some that they do not, as we do not restrict our semiring to being $\mathbb{N}$ or $\mathbb{Z}$. Moreover, the investigation in~\cite{Arenas}, by contrast to ours, concentrates on the study of classical counting classes for ordered structures, while we consider both ordered and arbitrary structures (provided, in the latter case, that the semiring is idempotent and commutative; examples include e.g.\ the max-plus- and min-plus-semirings). In the present article, the central aim is rather starting the study of weighted complexity classes via logic, and the corollaries characterizing classical complexity classes are obtained as interesting byproducts of the work. In this way, we are also meeting the challenge posed in~\cite[p.3]{Kos} of developing ``quantitative descriptive complexity theory based on weighted logics [$\dots$]
over some fairly general class of semirings''. Further work on the  model theory of weighted logics includes a Feferman--Vaught result~\cite{DrostePaul}, but the area remains largely unexplored despite being one of the open problems suggested in~\cite{DrosteGastin:07}.

\section{Weighted Turing machines}

In order to introduce the notion of a weighted Turing machine, first we need to define the kind of algebraic structures that will provide the weights, that is, semirings.

\begin{Def}[Semirings] A \emph{semiring} is a tuple $\mathcal{S}=\tuple{S,\splus,\smal,\szero,\sone}$, with operations addition $\splus$ and multiplication $\smal$ and constants $\szero$ and\/ $\sone$ such that 
\begin{itemize}
\item $\tuple{S,\splus,\szero}$ is a commutative monoid and $\tuple{S,\smal,\sone}$ is a monoid, 
\item multiplication distributes over addition, and 
\item $s \smal \szero = \szero \smal s = \szero$ for every $s \in S$. 
\end{itemize}
We say that $\mathcal{S}$ is \emph{commutative} if the monoid $\tuple{S,\smal,\sone}$ is commutative, and we say that $\mathcal{S}$ is \emph{idempotent} if the monoid $\tuple{S,\splus,\szero}$ is idempotent (that is, $s \splus s =s$ for each $s \in S$). \end{Def}

Some examples of semirings, including those that we will use in this paper, are the following:
\begin{itemize}
\item the \emph{Boolean semiring} $\B = \tuple{\{\bzero, \bone\}, \min, \max, \bzero, \bone}$,
\item any bounded distributive lattice $\tuple{L, \lor, \land, 0, 1}$,
\item the semiring of natural numbers $\tuple{\Nz, +, \cdot, 0, 1}$,
\item the semiring of extended natural numbers $\tuple{\Nz \cup \{ +\infty \}, +, \cdot, 0, 1}$ where $0 \cdot (+\infty) = 0$,
\item the ring of integers, $\tuple{\mathbb{Z},+,\cdot,0,1}$,
\item the ring of integers modulo $n$, $\tuple{\mathbb{Z}_n,+_n,\cdot_n,\overline{0},\overline{1}}$, for each $n\in \mathbb{N}$,
\item the field of rational numbers $\tuple{\Q, +, \cdot, 0, 1}$,
\item the \emph{max-plus} or \emph{arctic semiring} $\arct = \tuple{\R_+ \cup \{-\infty\}, \max, +, -\infty, 0}$, where $\R_+$ denotes the set of non-negative real numbers,
\item the restriction of the arctic semiring to the natural numbers $\mathbb{N}_{\max} = \tuple{\mathbb{N} \cup \{-\infty\}, \max, +, -\infty, 0}$,
\item the \emph{min-plus} or \emph{tropical semiring} $\trop = \tuple{\R_+ \cup \{+\infty\}, \min, +, +\infty, 0}$,
\item the restriction of the tropical semiring to the natural numbers $\mathbb{N}_{\min} = \tuple{\mathbb{N} \cup \{+\infty\}, \min, +, +\infty, 0}$,
\item the semiring $\mathcal{F}_* =\tuple{[0,1],\max,*,0,1}$ given by a t-norm $*$~\cite{Wie},
\item the semiring of finite languages $2^{\Sigma^*}_\mathrm{fin} =\tuple{2^{\Sigma^*}_\mathrm{fin},\cup,\cdot,\emptyset,\{\varepsilon\}}$, for an alphabet $\Sigma$,
\item the semiring $\mathcal{S}_{\max} = \tuple{\{0,1\}^* \cup \{-\infty\}, \max, \cdot, -\infty, \varepsilon}$ of binary words in which $\max$ is computed according to the {\em radix order} (for $x,y \in \{0,1\}^*$, $x \preceq y$ iff $|x| <|y|$ or $|x|=|y|$ and $x$ is smaller than or equal to $y$ in the lexicographic order) and $\max(x,- \infty) = \max(- \infty, x) = x$ for each $x$, $\cdot$ is the concatenation operation, and $x \cdot (- \infty) = (- \infty) \cdot x = - \infty$ for each $x$,
\item the semiring $\mathcal{S}_{\min} = \tuple{\{0,1\}^* \cup \{+\infty\}, \min, \cdot, +\infty, \varepsilon}$ analogous to the previous one. 
\end{itemize}

A notion from universal algebra (cf.~\cite{Bergman}) that we will make use of in defining some of the complexity classes below (e.g. $\mathsf{FP}[\mathcal{S}], \mathsf{FPSPACE}[\mathcal{S}]$ and $\mathsf{FPLOG}[\mathcal{S}]$) is the following:

\begin{Def}[Term algebra] Consider a semiring $\mathcal{S}=\tuple{S,\splus,\smal,\szero,\sone}$ and a subset $X\subseteq S$. Then the \emph{set of terms} $T(X)$  is the collection of  all well-formed strings that can be constructed using the symbols in $X$ and $\splus,\smal,\szero,\sone$ (in particular,  $\szero,\sone \in T(X)$). The \emph{term algebra} $\mathcal{T}(X)$ is the structure with universe $T(X)$ and operations $\splus',\smal'$ defined in the obvious way using the operations $\splus,\smal$
 from the semiring $\mathcal{S}$.
\end{Def}

\begin{Def}[Weighted Turing Machines]
Let $\mathcal{S}$ be a semiring and\/ $\Sigma$ an alphabet. A {\em weighted (or algebraic) Turing machine} over $\mathcal{S}$ and input alphabet $\Sigma$ is a septuple $\mathcal{M}= \tuple{Q, \Gamma,\Delta, \nu, q_0, F, \Box}$, where
\begin{itemize}
\item $Q$ is a nonempty finite set whose elements are called {\em states},
\item $\Gamma \supseteq \Sigma$ is an alphabet ({\em working alphabet}),
\item $\Delta \subseteq (Q \setminus F) \times \Gamma \times Q \times \Gamma  \times \{-1, 0, 1\}$ and its elements are called {\em transitions},
\item $\nu \colon \Delta \longrightarrow S$ is called a {\em transition weighting function}, $q_0 \in Q$ is called the {\em initial state}, $F \subseteq Q$ and its elements are called {\em accepting states}, and $\Box \in \Gamma \setminus \Sigma$ is the blank symbol.
\end{itemize}
We call $\mathcal{M}$ a \emph{Turing machine} if $\mathcal{S}$ is the Boolean semiring $\mathbb{B}$.
We call $\mathcal{M}$ \emph{deterministic} if for every pair $(p, a) \in Q \times \Gamma$,
there is at most one transition $(p,a,q,b,d) \in \Delta$.
\end{Def}

A {\em configuration} of $\mathcal{M}$ is a unique description of the machine’s state, contents of the working tape, and the position of the machine’s head. If $e = \tuple{p,c,q,d,t} \in \Delta$ is a transition and $C_1, C_2$ are configurations of $\mathcal{M}$, then we write $C_1 \longrightarrow_e C_2$ if $C_1$ is a configuration with state p and the head reading $c$, while $C_2$ is obtained from $C_1$ by changing state to $q$, rewriting the originally read symbol $c$ to $d$, and moving the head as prescribed by $t$. We write $C_1 \longrightarrow C_2$ if $C_1 \longrightarrow_e C_2$ for some $e\in \Delta$.

A {\em computation} of $\mathcal{M}$ is a word $\gamma = C_1e_1C_2e_2C_3\ldots C_ne_nC_{n+1}$ such that $C_1, \ldots, C_{n+1}$ are configurations of $\mathcal{M}$, $e_1, \ldots, e_n \in \Delta$, $C_k \longrightarrow_{e_k} C_{k+1}$ for each $k \in \{1, \ldots, n\}$, and $C_1$ is a configuration with state $q_0$ and the head at the leftmost
non-blank cell (if there is some). The {\em weight} of $\gamma$ is defined as $\nu(\gamma) := \nu(e_1)\nu(e_2)\ldots \nu(e_n)$. $\gamma$ is called an {\em accepting} computation if $C_{n+1}$ has an accepting state. We say that $\gamma$ is a computation on $w$ in $\Sigma^*$, and write $\Sigma(\gamma) = w$ if $C_1$ is a configuration with $w$ on the working tape. We denote the set of all computations of $\mathcal{M}$ by $C(\mathcal{M})$ and the set of all accepting computations by $A(\mathcal{M})$.

\begin{Convention}
From  now on we will assume that every Turing machine $\mathcal{M}$ is {\em finitely terminating}, that is,  the set $C_w(\mathcal{M}) = \{\gamma \in C(\mathcal{M}) \mid \Sigma(\gamma) = w\}$ is finite for each $w \in \Sigma^*$. In particular, the set $A_w(\mathcal{M}) = \{\gamma \in A(\mathcal{M}) \mid \Sigma(\gamma) = w\}$ is finite.
\end{Convention}

Thanks to the convention, we can introduce the following notion:

\begin{Def}[Behavior of a weighted Turing machine] Let $\mathcal{M}$  be a weighted Turing machine.
The {\em behavior} of $\mathcal{M}$ as the mapping $\|\mathcal{M}\| \colon \Sigma^* \longrightarrow S$ defined as
$$
 \|\mathcal{M}\|(w) := \sum_{\gamma \in A_w(\mathcal{M})} \nu(\gamma).
$$
We say that a series $\sigma\colon \Sigma^*  \longrightarrow S$ is \emph{recognized} by a weighted Turing machine $\mathcal{M}$ if\/ $\|\mathcal{M}\| = \sigma$.
\end{Def}

The definition of weighted Turing machine we have used here is exactly the same as that of algebraic Turing machines~\cite[Def.\ 5.1]{Damm} (see also~\cite{Kos}). Similarly, the notion of the behavior of the machine coincides. The semiring Turing machines of~\cite{Eiter1, Eiter}, by contrast,  differ in that they impose some conditions on  the allowed transitions~\cite[cf.\ Def.\ 12]{Eiter} (thus everything that can be done by a semiring Turing machine can be done by a weighted one, but the converse is not clear). Given distributivity of multiplication over addition, the notion of a semiring Turing machine function there~\cite[Def.\ 13]{Eiter} coincides with that of the behavior we use here.
All these definitions generalize the corresponding notions for weighted automata.

\section{Some weighted complexity classes}
Let $\mathcal{M}= \tuple{Q, \Gamma,\Delta, \nu, q_0, F, \Box}$ be a weighted Turing machine over $\mathcal{S}$ and $\Sigma$. For $w \in \Sigma^*$, we denote by $\mathsf{TIME}(\mathcal{M},w)$ the maximal length of a computation of $\mathcal{M}$ on $w$, and define, for $n \in \mathbb{N}$, $\mathsf{TIME}(\mathcal{M},n) := \max\{\mathsf{TIME}(\mathcal{M},w) : w \in \Sigma^*, |w|\leq n\}$.

For a function $f \colon \mathbb{N} \longrightarrow \mathbb{N}$, we denote by $\mathsf{SERIES}[S,\Sigma](f)$ the set of all series $\sigma$ such that $\sigma = \|\mathcal{M}\|$ for some weighted Turing machine $\mathcal{M}$ over $\mathcal{S}$ and $\Sigma$ with $\mathsf{TIME}(\mathcal{M},n) = O(f(n))$. Now we can define the complexity classes:
\begin{itemize}
\item[]$\mathsf{SERIES}[\mathcal{S}](f(n)):= \bigcup \{\mathsf{SERIES}[S,\Sigma](f(n)) : \Sigma$ is an alphabet$\}.$
\end{itemize}

\begin{Def}
Let $\mathcal{S}$ be a semiring. We define the following weighted complexity class  $$\mathsf{NP}[\mathcal{S}]:=\bigcup \{\mathsf{SERIES}[\mathcal{S}](n^k) : k \in \mathbb{N}\}.$$
\end{Def}

$\mathsf{NP}[\mathcal{S}]$ (cf.~\cite[Def.~4.1]{Kos}) coincides with the definition of the class $\mathcal{S}$-$\#\text{P}$ in~\cite[Def.~5.2]{Damm}. Furthermore, it is contained as a subclass in the similarly defined class $\text{NP}[\mathcal{R}]$ from~\cite[Def.~14]{Eiter} when $\mathcal{R}$ is a commutative semiring. Below (Proposition~\ref{countx}), we will actually show that this containment is proper, in the sense that  $\text{NP}[\mathcal{R}]$  will contain some series that are not in $\mathsf{NP}[\mathcal{S}]$.

\begin{Exm}\label{exm1}
Following~\cite[Prop.~5.3]{Damm} and~\cite[Examples~4.2--4.6]{Kos}, we can list some prominent instances of\/ $\mathsf{NP}[\mathcal{S}]$:
\begin{itemize}
\item the usual complexity class $\mathsf{NP}$, obtained when $\mathcal{S}=\mathbb{B}$ is the two-element Boolean semiring and each transition is weighted by $1$ (this is the standard way of representing a classical machine model in the weighted context),
\item the counting class $\#\mathsf{P}$~\cite{Valiant}, obtained when $\mathcal{S}= \tuple{\Nz,+,\cdot,0,1}$ is the semiring of natural numbers and each transition is weighted by $1$,
\item the complexity class $\bigoplus\mathsf{P}$~\cite{Papadimitriou-Zachos}, obtained when $\mathcal{S}= \tuple{\mathbb{Z}_2,+_2,\cdot_2,\overline{0},\overline{1}}$ is the finite field of two elements and each transition is weighted by $1$,
\item the class $\mathsf{GapP}$, closure of\/ $\#\mathsf{P}$ under subtraction~\cite{Fenner-Fortnow-Kurtz,Gupta}, obtained when $\mathcal{S}=\tuple{\mathbb{Z},+,\cdot,\overline{0},\overline{1}}$ is the ring of integers and transitions are weighted by $1$ and $-1$,
\item  the class $\mathsf{MOD}_q-\mathsf{P}$ (for $q\geq 2$)~\cite{Cai}, defined similarly to $\#\mathsf{P}$ but with respect to counting modulo $q$, obtained when $\mathcal{S}=\tuple{\mathbb{Z}_q,+_q,\cdot_q,\overline{0},\overline{1}}$ and transitions are weighted by $1$.
\end{itemize}
\end{Exm}

\begin{Exm}\label{exm2}
 Some further instances of\/ $\mathsf{NP}[\mathcal{S}]$, this time following~\cite[Examples 4.7--4.11]{Kos}, are:
\begin{itemize}

\item  the class $\mathsf{NP}[F_*]$ of all fuzzy languages realizable by fuzzy Turing machines~\cite{Wie} with t-norm $*$
in polynomial time, obtained when the semiring is $\mathcal{F}_* =\tuple{[0,1],\max,*,0,1}$ and the weights correspond to degrees of membership in the fuzzy language,

\item the class $\mathsf{NPMV}$
of all multivalued functions realized by nondeterministic polynomial-time transducer machines~\cite{Book-Long-Selman}, obtained when, given alphabets $\Sigma_1$ and $\Sigma_2$, the semiring is $\tuple{2^{\Sigma_2^*}_\mathrm{fin},\cup,\cdot,\emptyset,\{\varepsilon\}}$ and weighted Turing machines have input alphabet $\Sigma_1$,

\item the class of all multiset-valued functions computed by nondeterministic polynomial-time transducer machines with counting, obtained as in the previous example but using the free semiring $\tuple{\mathbb{N}\tuple{\Sigma^*_2},+,\cdot,0,1}$ instead,

\item the class $\mathsf{MaxP}\subseteq \mathsf{OptP}$ of problems in which the objective is to compute the value of a solution to an optimization problem in $\mathsf{NPO}$~\cite{Krentel}, obtained when the semiring is $\mathcal{S}_{\max}$, and the class $\mathsf{MinP}\subseteq \mathsf{OptP}$,  obtained when the semiring is $\mathcal{S}_{\min}$,

\item the class $\mathsf{MaxP}[[O(\log n)]]\subseteq \mathsf{OptP}[O(\log n)]$ of problems in which the objective is to compute the value of a solution to an optimization problem in $\mathsf{NPO}\,\mathsf{PB}$~\cite{Krentel}, obtained when the semiring is $\mathbb{N}_{\max}$, and\/ $\mathsf{MinP}[[O(\log n)]]\subseteq \mathsf{OptP}[O(\log n)]$, , obtained when the semiring is $\mathbb{N}_{\min}$.

\end{itemize}
\end{Exm}

\begin{Def}
We define the complexity class $\mathsf{FP}[\mathcal{S}]$ as
$$\mathsf{FP}[\mathcal{S}] := \bigcup_{\substack{\{0,1\} \subseteq G\subseteq_{\mathrm{fin}} S \\ \Sigma \ \text{is a finite alphabet}}} \mathsf{FP}[G, \Sigma] $$
where $\mathsf{FP}[G,\Sigma]$ is the set of all series $\sigma \colon \Sigma^* \longrightarrow \langle G \rangle$ (where $\langle G \rangle$ is the subsemiring of $\mathcal{S}$ generated by $G$) such that there is a constant $k \in \Nz$ and a deterministic polynomial-time Turing machine which outputs for every word $w\in \Sigma^*$ a  word of the form $\sum^{m_1}_{i_1=1} \prod^{n_1}_{j_1=1} \dotsb \sum^{m_k}_{i_k=1} \prod^{n_k}_{j_k=1} s_{i_1j_1 \dotsb i_k j_k}$ in the algebra of terms $T(G)$ in $S$ with value $\sigma(w)$ in $\mathcal{S}$.
Here, $T(G)$ is the smallest set of such of finite words which satisfies (1) $G \subseteq T(G)$ and (2) $(t_1 + t_2) \in T(G)$ and $(t_1 \cdot t_2) \in T(G)$ for every two terms $t_1, t_2 \in T(G)$; we abuse notation and omit parentheses whenever associativity permits.
\end{Def}

We note that this definition of $\mathsf{FP}[\mathcal{S}]$ differs from the one given in~\cite{Kos} in that we impose a bound on the number of alternations of the semiring operations.

\begin{Exm} If $\mathcal{S}=\mathbb{B}$ is the two-element Boolean semiring, then $\mathsf{FP}[\mathbb{B}]$ is just $\mathsf{P}$~\cite[Example 5.4]{Kos}. Observe that the terms output by  the machine in that example are already trivially of the form $\sum^n_{i=1} \prod^{m}_{j=1} s_{ij}$.
\end{Exm}

$\mathsf{FP}$ is to $\# \mathsf{P}$ what  $\mathsf{P}$ is to $\mathsf{NP}$. Thus, considering $\mathsf{NP}[\mathcal{S}]$ as a generalization of $\# \mathsf{P}$ (as it is done in~\cite{Damm}), the relationship between $\mathsf{FP}[\mathcal{S}]$ and $\mathsf{NP}[\mathcal{S}]$ is similar to that between $\mathsf{P}$ and $\mathsf{NP}$.

\begin{Exm}  If $\mathcal{S}=\mathbb{N}$ is the natural numbers semiring, then $\mathsf{FP}[\mathbb{N}]$ is just $\mathsf{FP}$~\cite[Example 5.5]{Kos}. As before, observe that the terms output by  the machine in that example are already of the form $\sum^n_{i=1} \prod^{m}_{j=1} s_{ij}$.

\end{Exm}

\begin{Def}
The  class   $\mathsf{FPLOG}[\mathcal{S}]$ is defined as $\mathsf{FP}[\mathcal{S}]$  except that we  allow the machine to have logarithmic space on the length of the input rather than polynomial time.
\end{Def}

\begin{Exm} If $\mathcal{S}=\mathbb{B}$, then $\mathsf{FPLOG}[\mathbb{B}]$ is just $\mathsf{DLOGSPACE}$. 
\end{Exm}

\begin{Exm}  If $\mathcal{S}=\mathbb{N}$, then $\mathsf{FPLOG}[\mathbb{N}]$ is just $\mathsf{FPLOG}$, which is defined as $\mathsf{FP}$ but allowing the machine to use logarithmic space on the size of the input (cf.~\cite{Gla}). 

\end{Exm}

\begin{Def}
The  class   $\mathsf{FPSPACE}[\mathcal{S}]$ is defined as $\mathsf{FP}[\mathcal{S}]$  except that we require allow the machine to have polynomial space on the length of the input rather than polynomial time.
\end{Def}

\begin{Exm} If $\mathcal{S}=\mathbb{B}$, then $\mathsf{FPSPACE}[\mathbb{B}]$ is just $\mathsf{PSPACE}$. 
\end{Exm}

\begin{Exm}  If $\mathcal{S}=\mathbb{N}$, then $\mathsf{FPSPACE}[\mathbb{N}]$ is just $\mathsf{FPSPACE}$ (\cite{Ladner}). 

\end{Exm}

\begin{Def}
The  class   $\mathsf{FPSPACE}_{poly}[\mathcal{S}]$ is defined as $\mathsf{FPSPACE}[\mathcal{S}]$  except that we require the word $\sum^n_{i=1} \prod^{m}_{j=1} s_{ij}$ to have length bounded by a polynomial.
Here, every semiring element is considered to have length $1$.
\end{Def}

\begin{Exm} If $\mathcal{S}=\mathbb{B}$, then $\mathsf{FPSPACE}_{poly}[\mathbb{B}]$ is just $\mathsf{PSPACE}$. 
\end{Exm}

\begin{Exm} If $\mathcal{S}=\mathbb{N}$, then $\mathsf{FPSPACE}_{poly}[\mathbb{N}]$ is just $\mathsf{FPSPACE}_{poly}$ (\cite{Ladner}). 
\end{Exm}

\section{Weighted logics}



A \emph{signature} (or \emph{vocabulary}) $\tau$ is a pair $\tuple{\rel{\tau}, \ar{\tau}}$ where $\rel{\tau}$ is a set of relation symbols and $\ar{\tau} \colon \rel{\tau} \longrightarrow \Np$ is the arity function. A \emph{$\tau$-structure}  $\mfa$ is a pair $\tuple{A, \inter{\mfa}}$ where $A$ is a set, called the \emph{universe of\/ $\mfa$}, and $\inter{\mfa}$ is an \emph{interpretation}, which maps every symbol $R \in \rel{\tau}$ to a set $R^\mfa \subseteq A^{\ar{\tau}(R)}$. We assume that each structure is \emph{finite}, that is, its universe is a finite set. A structure is called \emph{ordered} if it is given for a vocabulary $\tau \cup\{<\}$ where  $<$ in interpreted as a linear ordering with endpoints. By $\str{\tau}_<$ we denote the class of all finite ordered $\tau$-structures.

We provide a countable set $\mcv$ of first and second-order variables, where lower case letters like $x$ and $y$ denote first-order variables and capital letters like $X$ and $Y$ denote second-order variables. Each second-order variable $X$ comes with an associated arity, denoted by $\mathrm{ar} (X)$. We define first-order formulas $\beta$ over a signature $\tau$ and weighted first-order formulas $\varphi$ over $\tau$ and a semiring $\mathcal{S}$, respectively, by the grammars
\begin{align*}
\beta &\ddefas \false \mid R(x_1, \ldots, x_n) \mid \lnot \beta \mid \beta \lor \beta \mid \exists x.\beta\\
\varphi &\ddefas \beta \mid s \mid \varphi \fplus \varphi \mid \varphi \fmal \varphi \mid \bigfplusop x.\varphi \mid \bigfmalop x. \varphi,
\end{align*}
where $R \in \rel{\tau}$, $n = \ar{\tau}(R)$, $x, x_1, \ldots, x_n \in \mcv$ are first-order variables, and $s \in \Sr$. Likewise, we define second-order formulas $\beta$ over $\tau$ and weighted  second-order formulas $\varphi$ over $\tau$ and $\mathcal{S}$ through
\begin{align*}
\beta &\ddefas \false \mid R(x_1, \ldots, x_n) \mid X(x_1, \ldots, x_n) \mid \lnot \beta \mid \beta \lor \beta \mid \exists x.\beta \mid \exists X.\beta\\
\varphi &\ddefas \beta \mid s \mid \varphi \fplus \varphi \mid \varphi \fmal \varphi \mid \bigfplusop x.\varphi \mid \bigfmalop x. \varphi \mid \bigfplusop X. \varphi \mid \bigfmalop X. \varphi,
\end{align*}

with $R \in \rel{\tau}$, $n = \ar{\tau}(R)=\mathrm{ar} (X)$, $x, x_1, \ldots, x_n \in \mcv$ first-order variables, $X \in \mcv$ a second-order variable, and $s \in \Sr$. We also allow the usual abbreviations $\land$, $\forall$, $\to$, $\leftrightarrow$, and $\true$. By $\fo{\tau}$ and $\wfo{\tau,\Sr}$ we denote the sets of all first-order formulas over $\tau$ and all weighted first-order formulas over $\tau$ and $\mathcal{S}$, respectively, and by $\mso{\tau}$ and $\wmso{\tau,\Sr}$ we denote the sets of all second-order formulas over $\tau$ and all weighted second-order formulas over $\tau$ and $\mathcal{S}$, respectively.

The notion of \emph{free variables} is defined as usual, i.e., the operators $\exists, \forall, \bigfplusop$, and $\bigfmalop$ bind variables. We let $\free(\varphi)$ be the set of all free variables of $\varphi$. A formula $\varphi$ with $\free(\varphi) = \emptyset$ is called a \emph{sentence}. For a tuple $\vecg{\varphi} = \tuple{\varphi_1, \ldots, \varphi_n} \in \wmso{\tau, \Sr}^n$, we define $\free(\vecg{\varphi}) = \bigcup_{i=1}^n \free(\varphi_i)$.

We define the semantics of $\mathrm{SO}$ and $\mathrm{wSO}$ as follows.
Let $\tau$ be a signature, $\mfa = \tuple{A, \inter{\mfa}}$ a $\tau$-structure, and $\mcv$ a set of first and second-order variables. 
A $(\mcv, \mfa)$-assignment $\rho$ is a function $\rho \colon \mcv \longrightarrow A \cup \mcp(A)$ such that, whenever $x \in \mcv$ is a first-order variable and $\rho(x)$ is defined, we have $\rho(x) \in A$, and whenever $X \in \mcv$ is a second-order variable and $\rho(X)$ is defined, we have $\rho(X) \subseteq A^{\mathrm{ar} (X)}$.
 For a first-order variable, this restriction may cause the variable to become undefined.
Let $\dom(\rho)$ be the domain of $\rho$. For a first-order variable $x \in \mcv$ and an element $a \in A$, the \emph{update} $\rho[x \to a]$ is defined through $\dom(\rho[x \to a]) = \dom(\rho) \cup \{x\}$, $\rho[x \to a](\mcx) = \rho(\mcx)$ for all $\mcx \in \mcv\setminus\{x\}$, and $\rho[x \to a](x) = a$. For a second-order variable $X \in \mcv$ and a set $I \subseteq A$, the update $\rho[X \to I]$ is defined in a similar fashion. By $\mfa_\mcv$ we denote the set of all $(\mcv, \mfa)$-assignments.

For $\rho \in \mfa_\mcv$ and a formula $\beta \in \mso{\tau}$ the relation ``$\tuple{\mfa, \rho}$ satisfies $\beta$'', denoted by $\tuple{\mfa, \rho} \models \beta$, is defined as
\begin{align*}
\begin{aligned}
& \tuple{\mfa, \rho} \models \false && && \text{never holds}\\
& \tuple{\mfa, \rho} \models R(x_1, \ldots, x_n) && \Longleftrightarrow && x_1, \ldots, x_n \in \dom(\rho) \text{ and } (\rho(x_1), \ldots, \rho(x_n)) \in R^\mfa\\
& \tuple{\mfa, \rho} \models X(x_1,..., x_n) && \Longleftrightarrow && x_1,..., x_n, X \in \dom(\rho) \text{ and } \langle\rho(x_1), \dots,  \rho(x_n)\rangle \in \rho(X) \\
& \tuple{\mfa, \rho} \models \lnot \beta && \Longleftrightarrow && \tuple{\mfa, \rho} \models \beta \text{ does not hold}\\
& \tuple{\mfa, \rho} \models \beta_1 \lor \beta_2 && \Longleftrightarrow && \tuple{\mfa, \rho} \models \beta_1 \text{ or } \tuple{\mfa, \rho} \models \beta_2\\
& \tuple{\mfa, \rho} \models \exists x. \beta && \Longleftrightarrow && \tuple{\mfa, \rho[x \to a]} \models \beta \text{ for some } a \in A\\
& \tuple{\mfa, \rho} \models \exists X. \beta && \Longleftrightarrow && \tuple{\mfa, \rho[X \to I]} \models \beta \text{ for some } I \subseteq A.
\end{aligned}
\end{align*}

Let $\varphi \in \wmso{\tau,\Sr}$ and  $\mfa \in \str{\tau}_<$,  $a_1, \dots, a_k$ be an enumeration of the elements of $\mfa$ according to the ordering that serves as the interpretation of $<$, and for every integer $n$, let $I_1^n, \dots, I_{l_n}^n$  be an enumeration of the subsets of $A^n$ according to the lexicographic ordering induced by the interpretation of $<$. The \emph{(weighted) semantics} of $\varphi$ is a mapping $\ser{\varphi}(\mfa,\cdot) \colon \mfa_\mcv \longrightarrow \Sr$ inductively defined as
\begin{align*}
\begin{aligned}
&\ser{\beta}(\mfa, \rho) &&=&& \begin{cases} \sone & \text{if } \tuple{\mfa, \rho} \models \beta \\ \szero & \text{otherwise} \end{cases}\\
&\ser{s}(\mfa, \rho) &&=&& s\\
&\ser{\varphi_1 \fplus \varphi_2}(\mfa, \rho) &&=&& \ser{\varphi_1}(\mfa, \rho) \splus \ser{\varphi_2}(\mfa, \rho)\\
&\ser{\varphi_1 \fmal \varphi_2}(\mfa, \rho) &&=&& \ser{\varphi_1}(\mfa, \rho) \smal \ser{\varphi_2}(\mfa, \rho)\\
&\ser{\bigfplusop x. \varphi}(\mfa, \rho) &&=&& \bigsplus_{a \in A} \ser{\varphi}(\mfa, \rho[x \to a])\\
&\ser{\bigfmalop x. \varphi}(\mfa, \rho) &&=&&   \prod_{1\leq i \leq k} \ser{\varphi}(\mfa, \rho[x \to a_i]) \\
&\ser{\bigfplusop X. \varphi}(\mfa, \rho) &&=&& \bigsplus_{I \subseteq A^{\mathrm{ar} (X)}} \ser{\varphi}(\mfa, \rho[X \to I])\\
&\ser{\bigfmalop X. \varphi}(\mfa, \rho) &&=&&  \prod_{1 \leq i \leq l_{\mathrm{ar} (X)}} \ser{\varphi}(\mfa, \rho[X \to I_i^{\mathrm{ar} (X)}]).
\end{aligned}
\end{align*}
 Note that if the semiring is commutative, in the clauses of universal quantifiers,
the semantics is defined by using any order for the factors in the products.

We will usually identify a pair $\tuple{\mfa, \emptyset}$ (where $\emptyset$ is the empty mapping) with $\mfa$.
We will also refer to the following expansions of $\mathrm{FO}$:
\begin{itemize}
\item Transitive closure logic ($\mathrm{TC}$) is obtained by adding the following rule for building formulas: if $\varphi(\overline{x},\overline{y})$ is a formula with variables $\overline{x} = x_1, \ldots, x_k$ and $\overline{y} = y_1, \ldots, y_k$, and $\overline{u},\overline{v}$ are $k$-tuples of terms, then $[\mathbf{tc}_{\overline{x},\overline{y}}\,\varphi(\overline{x},\overline{y})](\overline{u},\overline{v})$ is also a formula, and its semantics is given as\\
$\mfa \models [\mathbf{tc}_{\overline{x},\overline{y}}\,\varphi(\overline{x},\overline{y})](\overline{a},\overline{b}) \,\, \Longleftrightarrow \,\,$  there exist an $n\geq 1$ and  $\overline{c_0}, \ldots, \overline{c_n} \in A^k$ such that $\overline{c_0} = \overline{a}$, $\overline{c_n} = \overline{b}$, and $\mfa \models \varphi(\overline{c_i},\overline{c_{i+1}})$ for each $i \in \{0,\ldots, n-1\}$.
\item Deterministic transitive closure logic ($\mathrm{DTC}$) is obtained by adding the following rule for building formulas: if $\varphi(\overline{x},\overline{y})$ is a formula with variables $\overline{x} = x_1, \ldots, x_k$ and $\overline{y} = y_1, \ldots, y_k$, and $\overline{u},\overline{v}$ are $k$-tuples of terms, then $[\mathbf{dtc}_{\overline{x},\overline{y}}\,\varphi(\overline{x},\overline{y})](\overline{u},\overline{v})$ is also a formula, and its semantics is defined by the equivalence $[\mathbf{dtc}_{\overline{x},\overline{y}}\,\varphi(\overline{x},\overline{y})](\overline{u},\overline{v}) \equiv [\mathbf{tc}_{\overline{x},\overline{y}}\,\varphi(\overline{x},\overline{y}) \land \forall z (\varphi(\overline{x},\overline{z}) \to \overline{y} = \overline{z})](\overline{u},\overline{v})$.
\item Least fixed-point logic ($\mathrm{LFP}$) is obtained by adding the following rules for building formulas: if $\varphi(R,\overline{x})$ is a formula of vocabulary $\tau \cup \{R\}$ with only positive occurrences of $R$, $\overline{x}$ is a tuple of variables, and $\overline{t}$ is a tuple of terms (both matching the arity of $R$), then $[\mathbf{lfp}\,R\overline{x}.\psi](\overline{t})$ and $[\mathbf{gfp}\,R\overline{x}.\psi](\overline{t})$ are also formulas. For their semantics, we need to define some auxiliary notions. The {\em update operator} $F_\psi\colon \mathcal{P}(A^k) \longrightarrow \mathcal{P}(A^k)$ is defined by $F_\psi(R):= \{\overline{a} \mid \tuple{\mfa,R}\models \psi(R,\overline{a})\}$ for any relation $R$, and it is monotone because $R$ occurs only positively in $\psi$. A {\em fixed point} of $F_\psi$ is a relation $R$ such that $F_\psi(R) = R$. Since $F_\psi$ is monotone, it has a least and a greatest fixed point (by Knaster--Tarski Theorem). The semantics is given by: $\mfa \models [\mathbf{lfp}\,R\overline{x}.\psi](\overline{t})$ iff $\overline{t}^\mfa$ is contained in the least fixed point of $F_\psi$ (analogously for $[\mathbf{gfp}\,R\overline{x}.\psi](\overline{t})$ and the greatest fixed point).
\item Partial fixed-point logic ($\mathrm{PFP}$) is obtained by adding the following rule for building formulas: if $\varphi(R,\overline{x})$ is a formula of vocabulary $\tau \cup \{R\}$, $\overline{x}$ is a tuple of variables, and $\overline{t}$ is a tuple of terms (both matching the arity of $R$), then $[\mathbf{pfp}\,R\overline{x}.\psi](\overline{t})$ is also a formula. For the semantics, we consider again the update operator (now not necessarily monotone) and the sequence of its finite stages: $R^0:=\emptyset$ and $R^{m+1}:=F_\psi(R^m)$. In a finite structure $\mfa$, the sequence either reaches a fixed point or it enters a cycle of period greater than one. We define the partial fixed point of $F_\psi$ as the fixed point reached in the former case, or as the empty set in the latter case. Now, the semantics is given by: $\mfa \models [\mathbf{pfp}\,R\overline{x}.\psi](\overline{t})$ iff $\overline{t}^\mfa$ is contained in the partial fixed point of $F_\psi$.

\item Inflationary fixed-point logic ($\mathrm{IFP}$) is obtained by adding the following rules for building formulas: if $\varphi(R,\overline{x})$ is a formula of vocabulary $\tau \cup \{R\}$, $\overline{x}$ is a tuple of variables, and $\overline{t}$ is a tuple of terms (both matching the arity of $R$), then $[\mathbf{ifp}\,R\overline{x}.\psi](\overline{t})$ is also a formula. For its semantics, we need to define some auxiliary notions. An operator $G:\mathcal{P}(B) \longrightarrow \mathcal{P}(B)$ is said to be \emph{inflationary} if $X\subseteq G(X)$ for all $X \in \mathcal{P}(B)$. With any operator $F:\mathcal{P}(B) \longrightarrow \mathcal{P}(B)$ one can associate an inflationary operator $G$ by setting $G(X) := X \cup F(X)$. Iterating $G$ gives a fixed point that we will called the \emph{inflationary fixed point} of $F$.
 The semantics is given by: $\mfa \models [\mathbf{ifp}\,R\overline{x}.\psi](\overline{t})$ iff $\overline{t}^\mfa$ is contained in the inflationary fixed point of $F_\psi$.
\end{itemize} 

The weighted version of each of these logics is defined analogously as in the case of $\mathrm{FO}$ and $\mathrm{SO}$ by expanding the logics $\mathrm{TC}$, $\mathrm{DTC}$, $\mathrm{LFP}$,  $\mathrm{PFP}$, and $\mathrm{IFP}$ with the same weighted constructs as given for $\mathrm{wFO}$ and $\mathrm{wSO}$.  By a famous result of Gurevich and Shelah~\cite{G-S}, on finite structures,  $\mathrm{LFP}$ coincides with $\mathrm{IFP}$ and thus their weighted versions, $\mathrm{wIFP}$ and $\mathrm{wLFP}$, as we have defined them here, will also coincide in expressive power.

\section{Logical characterizations of complexity classes}\label{complexitysec}

We are finally ready to present and prove the main results of the paper: the quantitative versions of several logical characterizations of prominent complexity classes. We may assume that every $\mathfrak{A}\in \str{\tau}_<$ is encoded by a string of $0$s and $1$s. For example, where $\mathfrak{A}= \tuple{A, R_1^\mathfrak{A}, \dots, R_j^\mathfrak{A}}$ with $|A|=n$ (and we may assume in fact that $A=\{0, \dots, n-1\}$) we might let $$\text{enc}(\mathfrak{A})= \text{enc}(R_1^\mathfrak{A}) \cdot \dots \cdot \text{enc}(R_j^\mathfrak{A})$$
where if $R_i^\mathfrak{A}$  is an $l$-ary relation, then $\text{enc}(R_i^\mathfrak{A})$ is a string of symbols of length $n^l$ with a $1$ in its $m$th position if the $m$th tuple of $n^l$ is in $R_i^\mathfrak{A}$ and a $0$ otherwise.

\begin{Def}\label{ch}
Consider a weighted logic $\mathrm{L}[\mathcal{S}]$ (with weights in a semiring $\mathcal{S}$) and a weighted complexity class $\mathcal{C}$, which is simply a collection of series. We say that $\mathrm{L}[\mathcal{S}]$ \emph{captures} $\mathcal{C}$ over ordered structures in the vocabulary $\tau=\{R_1, \dots, R_j\}$ if:
\begin{itemize}
\item[(1)] For every $\mathrm{L}[\mathcal{S}]$-formula $\phi$, there exists  $P\in \mathcal{C}$ such that $P(\text{enc}(\mathfrak{A})) = \|\phi\|(\mathfrak{A})$ for every finite ordered $\tau$-structure $\mathfrak{A}$, and
\item[(2)] For every $P\in \mathcal{C}$, there exists an  $\mathrm{L}[\mathcal{S}]$-formula $\phi$ such that $P(\text{enc}(\mathfrak{A})) = \|\phi\|(\mathfrak{A})$ for every finite ordered $\tau$-structure $\mathfrak{A}$.
\end{itemize}

\end{Def}

The seminal Fagin's Theorem characterizes $\mathsf{NP}$ for ordered structures by existential second-order logic. Our goal is to present a weighted version of this result with arbitrary semirings as weight structures.
Whereas in the classical setting one obtains an equivalence between the existence of runs of a
Turing machine vs.\ the satisfiability of an existential logical formula, in the weighted setting we have
to derive a one-to-one correspondence between the runs of a Turing machine and satisfying assignments
for the formulas. Moreover, due to the absence of a natural negation function in the semiring,
here, beyond the classical setting, we need conjunctions and universal quantifications.
For weighted finite automata over words, in~\cite{DrosteGastin:07} weighted conjunction and universal quantification
turned out to be too powerful in general and had to be restricted. Surprisingly, here we do not need
these restrictions, but we can show the expressive equivalence between weighted polynomial-time Turing machines
and the full weighted existential second-order logic. Moreover, we do not need commutativity of the multiplication of $\mathcal{S}$
(essential in~\cite{DrosteGastin:07}), but can develop our characterization for arbitrary, also non-commutative, semirings $\mathcal{S}$.
This is due to new constructions, in this setting, for the involved weighted Turing machines. By $\mathrm{wESO}$ we mean the fragment of $\mathrm{wSO}$ where the only  second-order quantifiers appear at the beginning of the formula and are additive existential.


\begin{Thm}[Weighted Fagin's theorem]\label{thm:wfagin}
Let $\mathcal{S}$ be a semiring.
\begin{itemize}
\item[(i)] The logic $\mathrm{wESO}[\mathcal{S}]$ captures $\mathsf{NP}[\mathcal{S}]$ over \emph{ordered} finite structures in the vocabulary $\tau=\{R_1, \dots, R_j\}$. 
\item[(ii)] Assume that $\mathcal{S}$ is idempotent and commutative. Then, the logic $\mathrm{wESO}[\mathcal{S}]$  captures $\mathsf{NP}[\mathcal{S}]$ over \emph{all} finite structures in the vocabulary $\tau=\{R_1, \dots, R_j\}$.

\end{itemize}
\end{Thm}

Let us indicate some ideas for the proof. For ($i$),
first, for a given $\mathrm{wESO}$-formula $\phi$, we have to construct an $\mathsf{NP}$ Turing machine $\mathcal{M}$ with $\|\phi\| = \|\mathcal{M}\|$. For first-order formulas $\beta$, we can follow the classical proof. Regarding weighted formulas $\phi$, let us comment on the interesting cases. For weighted conjunctions and universal quantifications, we employ new constructions. Since we are dealing with Turing machines, we can execute weighted Turing machines for the components successively, by saving the word and using transitions of weight $1$ in a deterministic way to restore the initial tape configuration. We can show, using the distributivity of the semiring, that the constructed nondeterministic machine $\mathcal{M}$ computes precisely the values prescribed by the semantics of the weighted conjunction or the weighted universal quantifications, respectively.

Second, given a weighted NP Turing machine $\mathcal{M}$, by the assumption on its polynomial time usage, we construct a second-order formula $\psi$ reflecting the accepting computation paths of $\mathcal{M}$ and their employed transitions in a one-to-one correspondence; this enables us to incorporate the weights of the transitions by means of constants in the formula. The order is used for the construction of the formula such that the interpretation of weighted universal quantification reflects precisely the weights of the computation sequences of the given Turing machine.

For ($ii$), the order in universal quantifications now is taken care of by the commutativity of the multiplication, and the existence of an order is taken care of by an additional existential second-order quantification where idempotency of $\mathcal{S}$ implies that we obtain the same value.

From Theorem~\ref{thm:wfagin} and Examples~\ref{exm1} and~\ref{exm2}, we immediately obtain the following corollary:

\begin{Cor} For ordered structures in a finite vocabulary $\tau=\{R_1, \dots, R_j\}$, we have that:
\begin{itemize}
\item[(1)]  $\mathrm{wESO}[\mathbb{B}]$  \emph{captures} $\mathsf{NP}$ (originally proved in~\cite{Fagin}).
\item[(2)]  $\mathrm{wESO}[\mathbb{N}]$  \emph{captures} $\# \mathsf{P}$ (originally proved in~\cite{Arenas} and~\cite{Saluja:95}).
\item[(3)]  $\mathrm{wESO}[\mathbb{Z}]$  \emph{captures} $ \mathsf{GapP}$ (originally proved in~\cite{Arenas}).
\item[(4)]  $\mathrm{wESO}[\mathcal{S}_{\max}]$  (respectively, $\mathrm{wESO}[\mathcal{S}_{\min}]$). \emph{captures} $ \mathsf{MaxP}$ ($ \mathsf{MinP}$) (originally proved in~\cite{Arenas}).
\item[(5)]  $\mathrm{wESO}[\mathbb{Z}_2]$  \emph{captures} $\bigoplus \mathsf{P}$.
\item[(6)]  $\mathrm{wESO}[\mathbb{Z}_q]$  \emph{captures} $\mathsf{MOD}_q-\mathsf{P}$.
\item[(7)] 
 $\mathrm{wESO}[\mathbb{N}_{\max}]$  (respectively, $\mathrm{wESO}[\mathbb{N}_{\min}]$) \emph{captures} $ \mathsf{MaxP}[O(\log n)]$ ($ \mathsf{MinP}[O(\log n)]$). 
\item[(8)]  $\mathrm{wESO}[\mathcal{F}_*]$  \emph{captures} the class of all fuzzy languages realizable by fuzzy Turing machines with t-norm $*$
in polynomial time.
\item[(9)]$\mathrm{wESO}[2^{\Sigma^*_2}_\mathrm{fin}]$ \emph{captures} $\mathsf{NPMV}$.
\item[(10)] $\mathrm{wESO}[\mathbb{N}\tuple{\Sigma^*_2}]$  \emph{captures} the class of all multiset-valued functions computed by nondeterministic polynomial-time transducer machines with counting.
\end{itemize}
\end{Cor}

\begin{Rmk}
It is worth observing that the proofs of (2)-(4) in~\cite{Arenas} (Prop.~4.2, Cor.~4.8, and Thm.~4.10) are (as expected) different from ours. Our argument works in all those cases but neither of the three arguments given in~\cite{Arenas}  works for our more general setting.
\end{Rmk}

Our next application of the weighted Fagin's theorem consist in providing a natural computational problem complete for the class $\mathsf{NP}[\mathcal{S}]$ for certain semirings $\mathcal{S}$. Given a semiring $\mathcal{S}$, alphabets $\Sigma_1, \Sigma_2$, and series $\sigma_1 \in S\llangle \Sigma_1^*\rrangle$ and $\sigma_2 \in S\llangle \Sigma_2^*\rrangle$, we say that $\sigma_1$ is {\em polynomially many-one reducible} to $\sigma_2$ ($\sigma_1 \leq_m \sigma_2$, in symbols) if there is an $f\colon \Sigma_1^* \longrightarrow \Sigma_2^*$ computable deterministically in polynomial time such that $\tuple{\sigma_2, f(w)} = \tuple{\sigma_1,w}$ for each $w \in \Sigma_1^*$. A series $\sigma \in S\llangle \Sigma^*\rrangle$ is said to be {\em $\mathsf{NP}[\mathcal{S}]$-hard} if $\sigma' \leq_m \sigma$ for all $\sigma'$ in $\mathsf{NP}[\mathcal{S}]$. If, moreover, $\sigma$ belongs to $\mathsf{NP}[\mathcal{S}]$, then it is called {\em $\mathsf{NP}[\mathcal{S}]$-complete}.

Fix an infinite set $X$. The language of the weighted propositional logic over a finitely generated semiring $\mathcal{S}$ is built from $X$ as propositional variables, elements of $S$ as truth-constants, and logical connectives $\land, \lor, \neg$ (where negation is only applied to propositional variables). Let $\mathtt{Fmla}[\mathcal{S}]$ be the set of all formulas.
A {\em truth assignment} is a mapping $V \colon X \longrightarrow \{0,1\}$ extended to all formulas in the following way:
\begin{enumerate}
\item For each propositional variable $X$, let $\overline{V}(x) := V(x)$ and $\overline{V}(\neg x) := 1$ iff $V (x) = 0$. Moreover, let $\overline{V}(a) := a$ for each $a \in S$.
\item $\overline{V}(\varphi \lor \psi) := \overline{V}(\varphi) + \overline{V}(\psi)$ and $\overline{V}(\varphi \land \psi) := \overline{V}(\varphi) \cdot \overline{V}(\psi)$.
\end{enumerate}

For each formula $\varphi\in\mathtt{Fmla}[\mathcal{S}]$, let $X_\varphi$ be the set of propositional variables that occur in $\varphi$. Clearly, $\overline{V}(\varphi)$ depends only the values of $V$ on $X_\varphi$. The `problem' $\text{SAT}[\mathcal{S}]$ is the series $\sigma\colon\mathtt{Fmla}[\mathcal{S}] \longrightarrow S$ defined as follows: $\text{SAT}[\mathcal{S}](\varphi) = \sum_{V \in \{0,1\}^{X_\varphi}} \overline{V}(\varphi)$.

The following corollary  of our weighted version of Fagin's theorem has also appeared as~\cite[Thm.~6.3]{Kos} with a direct proof. Our  proof generalizes the reasoning for the Boolean case in~\cite{G-K-L}.

\begin{Cor}[Weighted Cook--Levin's theorem]
Let $\mathcal{S}$ be a finitely generated semiring. Then, $\mathsf{SAT}[\mathcal{S}]$ is $\mathsf{NP}[\mathcal{S}]$-complete.
\end{Cor}

Now it is natural to wonder what happens with other well-known descriptive complexity results. In the reminder of this section we will tackle a few more of these. We start with the Immerman--Vardi's theorem, a result that first appeared in the Boolean case in the papers~\cite{Immerman1, Vardi}. Our own approach is inspired by~\cite[Thm.~4.4]{Arenas} where a version of the result for the counting complexity class $\mathsf{FP}$ is provided using a weighted logic with the semiring $\mathbb{N}$. We must observe, however, that our proof is a generalization of that in~\cite{Arenas} that works for all semirings and not only $\mathbb{N}$.

\begin{Thm}[Weighted Immerman--Vardi's theorem]\label{I-V} The logic $\mathrm{wLFP}[\mathcal{S}]$ (with weights in a semiring $\mathcal{S}$) \emph{captures} $\mathsf{FP}[\mathcal{S}]$ over ordered structures in the vocabulary $\tau=\{R_1, \dots, R_j\}$. 

\end{Thm}

\begin{Cor} For ordered structures in a finite vocabulary $\tau=\{R_1, \dots, R_j\}$, we have that:
\begin{itemize}
\item[(1)]  $\mathrm{wLFP}[\mathbb{B}]$   \emph{captures} $\mathsf{P}$ (originally  proved in~\cite{Immerman1, Vardi}).
\item[(2)]  $\mathrm{wLFP}[\mathbb{N}] $  \emph{captures} $\mathsf{FP}$ (originally proved in~\cite{Arenas}).

\end{itemize}

\end{Cor}

\begin{Rmk} Observe that using second-order Horn logic  (which is known to capture $\mathsf{P}$~\cite{G}) instead of least fixed-point logic,  would not work for us, as in the weighted version one can encode a $\# \mathsf{P}$-complete problem (namely $\# \mathrm {HORNSAT}$). This was already noted in~\cite{Arenas}.

\end{Rmk}

In the next result, $\mathrm{wPFP}[\mathcal{S}] + \{\prod X, \sum X\}$ will denote the logic that is obtained from $\mathrm{wPFP}[\mathcal{S}]$ by the addition of the second-order quantitative quantifiers $\prod X$ and $\sum X$. Clearly, when  $\mathcal{S} = \mathbb{B}$, this is the same as second-order logic with partial fixed points. The Boolean counterpart of Theorem~\ref{space}, namely that second-order logic extended with partial fixed points characterizes $\mathsf{PSPACE}$ is folklore, but a proof can be found in~\cite[Thm.~4]{Richerby}. The classical argument also uses the result for partial fixed-point logic in~\cite{Vianu, Vardi} stating that the logic characterizes $\mathsf{PSPACE}$ over ordered structures.

\begin{Thm} \label{space}
The logic $\mathrm{wPFP}[\mathcal{S}] + \{\prod X, \sum X\}$ (with weights in a semiring $\mathcal{S}$) \emph{captures} $\mathsf{FPSPACE}[\mathcal{S}]$ over ordered structures in the vocabulary $\tau=\{R_1, \dots, R_j\}$. 
\end{Thm}

\begin{Cor} For ordered structures in a finite vocabulary $\tau=\{R_1, \dots, R_j\}$, we have that:
\begin{itemize}
\item[(1)]  $\mathrm{wPFP}[\mathbb{B}] + \{\prod X, \sum X\}$   \emph{captures} $\mathsf{PSPACE}$ (folklore, cf.~\cite{Richerby}).
\item[(2)]  $\mathrm{wPFP}[\mathbb{N}] + \{\prod X, \sum X\}$  \emph{captures} $\mathsf{FPSPACE}$ (originally proved in~\cite{Arenas}).
\end{itemize}
\end{Cor}

\begin{Thm}\label{spacepoly}
The logic $\mathrm{wPFP}[\mathcal{S}]$ (with weights in a semiring $\mathcal{S}$) \emph{captures} $\mathsf{FPSPACE}_{poly}[\mathcal{S}]$ over ordered structures in the vocabulary $\tau=\{R_1, \dots, R_j\}$. 
\end{Thm}

\begin{Cor} For ordered structures in a finite vocabulary $\tau=\{R_1, \dots, R_j\}$, we have that:
\begin{itemize}
\item[(1)]  $\mathrm{wPFP}[\mathbb{B}]$   \emph{captures} $\mathsf{PSPACE}$ (originally proved in~\cite{Vianu, Vardi}).
\item[(2)]  $\mathrm{wPFP}[\mathbb{N}]$  \emph{captures} $\mathsf{FPSPACE}_{poly}$ (originally proved in~\cite{Arenas}).
\end{itemize}
\end{Cor}

\begin{Thm}\label{logi}
The logic $\mathrm{wDTC}[\mathcal{S}]$ (with weights in a semiring $\mathcal{S}$) \emph{captures} $\mathsf{FPLOG}[\mathcal{S}]$ over ordered structures in the vocabulary $\tau=\{R_1, \dots, R_j\}$. 
\end{Thm}

\begin{Cor} For ordered structures in a finite vocabulary $\tau=\{R_1, \dots, R_j\}$, we have that:
\begin{itemize}
\item[(1)]  $\mathrm{wDTC}[\mathbb{B}]$   \emph{captures} $\mathsf{DLOGSPACE}$ (originally proved in~\cite{Immerman1}).
\item[(2)]  $\mathrm{wDTC}[\mathbb{N}]$  \emph{captures} $\mathsf{FPLOG}$.
\end{itemize}
\end{Cor}

To end the  present section, we  address the general and interesting open problem suggested in~\cite{Eiter} regarding a Fagin theorem that characterizes the class $\text{NP}[\mathcal{R}]$ from~\cite[Def.~14]{Eiter}.
We begin by observing that for the machine model in~\cite[Def.~12]{Eiter}, Fagin's theorem will fail if the logic considered is $\mathrm{wESO}$. This is essentially due to the fact that semiring Turing machines allow for an infinite number of transitions.  However, such a large set of transitions, is only actually needed when there are infinitely many semiring values in the input words.

\begin{Pro}\label{countx} Let $\mathcal{R}$ be a commutative semiring. There is a series $P\in \text{\emph{NP}}[\mathcal{R}]$ such that for no $\varphi \in \mathrm{wESO}$, $|| \varphi|| = P$.

\end{Pro}

Thus one might reasonably further ask what kind of logic would capture $\text{NP}[\mathcal{R}] $. Observe that an obvious challenge here is that in the proof of Fagin's theorem at some point we need to encode in the logic by means of a sentence involving a long (but finite) disjunction what the legal transitions of our machine are. Consequently, in the presence of infinitely many transitions, it is not clear how to achieve a Fagin-style characterization in a finitary language as before.   

By contrast to the above situation, we might ask a more restricted question if what we are doing  is trying to capture $\text{NP}[\mathcal{R}] $ over the class of all finite ordered structures. Recall that we are considering finite structures to be given via their binary encodings and thus the relevant series in  $\text{NP}[\mathcal{R}] $ are those that take as input merely binary strings. These series are not computed by 
SRTMs that involve infinitely many transitions because the input words do not involve semiring values. So  let us consider now the modification of~\cite[Def.~12]{Eiter} that only allows semiring Turing machines to come with a finite set of transitions. In this case we will easily see that their machine model coincides with ours.

\begin{Pro} Let $\mathcal{R}$ be a commutative semiring and allow only finitely many transitions in a semiring Turing machine. Then $\text{\emph{NP}}[\mathcal{R}] =\mathsf{NP}[\mathcal{R}]$, i.e.\ the NP class in the sense of~\cite{Eiter} coincides with the NP class in our sense. 

\end{Pro}

\section{Conclusions and further work}

In this paper, we have established a few central results in weighted descriptive complexity, providing quantitative versions of Fagin's theorem and the Immerman--Vardi's theorem, among other logical characterizations of complexity classes. We also plan to extend our weighted Fagin's theorem to the even larger class of valuation monoids containing all semirings
and supporting average calculations by the theory developed in~\cite{GM} for weighted finite automata over words and weighted $\mathrm{EMSO}$ logic.

Furthermore, in future work, we  aim to characterize further weighted complexity classes. For example,  in the definition of $\mathsf{NP}[\mathcal{S}]$, by changing the requirement about polynomial time to logarithmic space on the size of the input, we can obtain a weighted complexity class that generalizes the classical counting class $\# \mathsf{L}$. The latter has been characterized by means of a logic weighted on the semiring $\mathbb{N}$ in~\cite[Thm.~6.4]{Arenas}. We suspect that this work can be generalized.

\appendix
\section{Appendix}

In this section we include some examples, remarks and proofs that complete the body of the text.

We recall now a standard description of cliques in graphs by second-order logic; we will use this in our subsequent examples for weighted structural properties.
\setcounter{Thm}{10}
\begin{Exm}\label{ex:check_clique}
Let $\tau$ is the signature of a graph, i.e., $\rel{\tau} = \{\edge\}$ with $\edge$ binary. We call a graph $\mfg \in \str{\tau}$ \emph{undirected} if its interpretation of $\edge$ is a symmetric relation on the universe of $\mfg$. For every undirected graph $\mfg \in \str{\tau}$ and a subset $I$ of its universe, we can check whether the nodes from $I$ form a clique in $\mfg$ using the $\mathrm{SO}$-formula
\begin{align*}
\clique(X) \defas \forall x \forall y \Big( \big( Xx \land  Xy \land x \neq y\big) \to \edge(x,y) \Big).
\end{align*}
Here, the formula $x \neq y$ is an abbreviation for $\exists Y (Yy \land \lnot (Yx))$. We have that $\tuple{\mfg, [X \to I]}$ satisfies $\clique(X)$ if and only if $I$ is a clique in $\mfg$.
\end{Exm}

We  give next some examples of how weighted formulas can be interpreted.

\begin{Exm}
If $\mathcal{S} = \B$ is the two-element Boolean semiring, we obtain classical logic.
\end{Exm}

\begin{Exm}
Using the arctic semiring $\arct = \tuple{\R_+ \cup \{-\infty\}, \max, +, -\infty, 0}$, we can describe the size of the largest clique in a graph as follows. We reuse the signature $\tau$ of a graph and the $\mathrm{SO}$-formula $\clique(X)$ from Example~\ref{ex:check_clique} and define a $\mathrm{wSO}$-formula as follows.
\begin{align*}
\varphi \defas \bigfplus X. \Big( \clique(X) \fmal \bigfmal x. \big( 0 \fplus (1 \fmal Xx) \big) \Big)
\end{align*}
Then, for every undirected graph $\mfg \in \str{\tau}$, we have that $\ser{\varphi}(\mfg)$ is the size of the largest clique in $\mfg$.
\end{Exm}

\begin{Exm}
Assume that $\mathcal{S} = \tuple{\Q, +, \cdot, 0, 1}$ is the field of rational numbers and that $\tau$ is the signature from the previous example.
Then, for every fixed $n \in \Np$, we can count the number of $n$-cliques of an undirected graph $\mfg \in \str{\tau}$ using the $\mathrm{wSO}$-formula
\begin{align*}
\varphi_n \defas \frac{1}{n!} \fmal \bigfplusop x_1 \ldots \bigfplusop x_n. \bigwedge_{i \neq j} \Big( (x_i \neq x_j) \land \edge(x_i, x_j) \Big).
\end{align*}
Here, $x_i \neq x_j$ again is an abbreviation for $\exists Y (x_j \in Y \land \lnot (x_i \in Y))$.
\end{Exm}

\begin{Exm}
We consider the \emph{minimum cut} of directed acyclic graphs. For this, we interpret these graphs as \emph{flow networks} in the following way. Every vertex which does not have a predecessor is considered a \emph{source}, every vertex without successors is considered a \emph{drain}, and every edge is assumed to have a capacity of $1$. Let $G = \tuple{V, E}$ be a directed acyclic graph where $V$ is the set of vertices and $E \subseteq V \times V$ the set of edges. A cut $\tuple{S,D}$ of $G$ is a partition of $V$, i.e., $S \cup D = V$ and $S \cap D = \emptyset$, such that all sources of $G$ are in $S$, and all drains of $G$ are in $D$. The \emph{minimum cut} of $G$ is the smallest number $|E \cap (S \times D)|$ such that $\tuple{S, D}$ is a cut of $G$.

We can express the minimum cut of directed acyclic graphs by a weighted formula as follows. We let $\tau$ be the signature from the previous two examples and as our semiring, we choose the tropical semiring  $\trop = \tuple{\R_+ \cup \{+\infty\}, \min, +, +\infty, 0}$. Then, using the abbreviation
\begin{align*}
\cut(X,Y) \defas \forall x. \Big( (Xx \leftrightarrow \lnot (Yx)) \land (\exists y. \edge(y,x) \lor Xx) \land (\exists y. \edge(x,y) \lor Yx) \Big),
\end{align*}
we can express the minimum cut of a directed acyclic graph $\mfg \in \str{\tau}$ using the formula
\begin{align*}
\varphi \defas \bigfplusop X. \bigfplusop Y. \Big( \cut(X, Y) \fmal \bigfmalop x. \bigfmalop y. (1 \fplus \lnot(Xx \land Yy \land \edge(x,y))) \Big).
\end{align*}
\end{Exm}

\begin{Exm}[cf.~\cite{DrosteGastin:07}]\label{ex:countexample}
Let $\mathcal{S} = \tuple{\Nz,+,\cdot,0,1}$ be the semiring of natural numbers and let $\varphi \in \wmso{\tau,\mathcal{S}}$ be a formula which does not contain any constants $s \in \Nz$.
Then, we may understand $\ser{\varphi}(\mfa, \rho)$ as the number of proofs we have that $\tuple{\mfa, \rho}$ satisfies $\varphi$ assuming that we interpret the weighted operators in the following way. 
For Boolean formulas, we simply consider satisfaction to give us one proof, and otherwise we have no proof. 
The sum $\ser{\varphi_1 \fplus \varphi_2}$ is the number of proofs we have that $\varphi_1 \lor \varphi_2$ is true. 
This says that, if we have $n$ proofs for $\varphi_1$ and $m$ proofs for $\varphi_2$, then we interpret this as having $n + m$ proofs for the fact that $\varphi_1 \lor \varphi_2$ is true. Likewise, we interpret the product $\ser{\varphi_1 \fmal \varphi_2}$ as the number of proofs we have that $\varphi_1 \land \varphi_2$ is true. Similar interpretations apply for the weighted quantifiers.
\end{Exm}

\setcounter{Thm}{11}

\begin{Thm}[Weighted Fagin's theorem]
Let $\mathcal{S}$ be a semiring.
\begin{itemize}
\item[(i)] The logic $\mathrm{wESO}[\mathcal{S}]$ captures $\mathsf{NP}[\mathcal{S}]$ over \emph{ordered} finite structures in the vocabulary $\tau=\{R_1, \dots, R_j\}$. 
\item[(ii)] Assume that $\mathcal{S}$ is idempotent and commutative. Then, the logic $\mathrm{wESO}[\mathcal{S}]$  captures $\mathsf{NP}[\mathcal{S}]$ over all finite structures in the vocabulary $\tau=\{R_1, \dots, R_j\}$.

\end{itemize}
\end{Thm}

\begin{proof}

$(i)$: 
In order to establish (1) from Definition~\ref{ch},
we construct, for every $\mathrm{wESO}$-formula $\phi$, an $\mathsf{NP}$ Turing machine $\mathcal{M}$ with $\|\phi\| = \|\mathcal{M}\|$. If $v_1, \ldots, v_m$ are the free variables of $\phi$, we encode every input structure $\mathfrak{A}= \tuple{A, R_1^\mathfrak{A}, \dots, R_j^\mathfrak{A}}$ and free variable assignments $v_1 = a_1, \dotsc, v_m = a_m$ for $\mathcal{M}$ by
  \[\text{enc}(\mathfrak{A})= 
                              \text{enc}(R_1^\mathfrak{A}) \cdot \dots \cdot \text{enc}(R_j^\mathfrak{A})
                              \text{enc}(a_1) \cdot \dots \cdot \text{enc}(a_m), \]
  where $n = |A|$ and $\text{enc}(a_i) = \text{enc}(\{a_i\})$ if $v_i$ is a first-order variable.
  Let us assume that $A = \{0, \dotsc, n-1\}$.
  We proceed by induction on the structure of formulas
  and begin by showing that for every Boolean formula $\beta$,
  there exists a deterministic polynomial time Turing machine $\mathcal{M}$
  such that $\tuple{\mathfrak{A}, a_1, \ldots, a_m} \models \beta$ iff $\mathcal{M}$
  accepts $\tuple{\mathfrak{A}, a_1, \ldots, a_m}$,
  see also~\cite[Proposition 6.6]{Libkin}.
  \begin{itemize}
    \item If $\beta = R(x_1, \dotsc, x_l)$ or $\beta = X(x_1, \dotsc, x_l)$,
          we construct a deterministic Turing machine which
          deterministically checks if the bit corresponding to $\tuple{a_1, \ldots, a_l}$
          is 1 in $\text{enc}(R)$ or $\text{enc}(X)$, respectively.
          More precisely, we check whether the $m$'th symbol of the encoding is 1 for
          $m = \sum_{i = 1}^l n^{l-i} a_i$.
    \item If $\beta = \alpha \lor \gamma$,
          we let $\mathcal{M}_\alpha$ and $\mathcal{M}_\gamma$ be
          the deterministic polynomial time Turing machines for $\alpha$ and $\gamma$, respectively.
          We construct a Turing machine for $\beta$ which
          (1) simulates $\mathcal{M}_\alpha$ and accepts if the simulation accepts,
          (2) otherwise continues to simulate $\mathcal{M}_\gamma$ and
          accepts if the simulation accepts and
          (3) otherwise rejects the input.
    \item If $\beta = \lnot \alpha$,
          we let $\mathcal{M}_\alpha$ be
          the deterministic polynomial time Turing machine for $\alpha$.
          We construct a Turing machine for $\beta$ which simulates
          $\mathcal{M}_\alpha$ and accepts if the simulation rejects the input.
    \item If $\beta = \exists x. \alpha$,
          we let $\mathcal{M}_\alpha$ be
          the deterministic polynomial time Turing machine for $\alpha$.
          We construct a Turing machine for $\beta$ which iterates
          over all elements $a \in A$ and simulates
          $\mathcal{M}_\alpha$ on the input $\tuple{\mathfrak{A}, a, a_1, \dotsc, a_l}$
          and accepts once a simulation accepts.
          If no simulation accepts, the input is rejected.
          In total, we simulate $\mathcal{M}_\alpha$ at most $|A|$ times.
  \end{itemize}
  All of these constructions produce deterministic polynomial time Turing machines.
  Let $\mathcal{M}_\beta$ be the Turing machine constructed for $\beta$.
  We obtain a weighted Turing machine $\mathcal{M}$ from $\mathcal{M}_\beta$
  with $\|\beta\| = \|\mathcal{M}\|$ by defining the weight of every transition
  by $1$.
  Note that as $\mathcal{M}_\beta$ is deterministic,
  there is exactly one run with weight $1$ in $\mathcal{M}$
  for every input accepted by $\mathcal{M}_\beta$.
  
  We continue with the weighted formulas $\phi$.
  The case $\phi = \beta$ is already covered.
  For $\phi = s$, we construct Turing machine which
  accepts in a single transition with weight $s$ to a final state.
  \begin{itemize}
    \item For $\phi = \psi \oplus \zeta$,
          we let $\mathcal{M}_\psi$ and $\mathcal{M}_\zeta$ be
          the weighted nondeterministic Turing machines for $\psi$ and $\zeta$, respectively.
          We construct a Turing machine $\mathcal{M}$ for $\phi$ which
          nondeterministically simulates either $\mathcal{M}_\psi$ or $\mathcal{M}_\zeta$
          on the input.
          The simulations are started by a single transition of weight $1$ from the new initial state
          into the initial state of either $\mathcal{M}_\psi$ or $\mathcal{M}_\zeta$.
          Thus, every run of either $\mathcal{M}_\psi$ or $\mathcal{M}_\zeta$
          on the input is simulated by exactly one run of $\mathcal{M}$.
    \item For $\phi = \psi \fmal \zeta$,
          we let $\mathcal{M}_\psi$ and $\mathcal{M}_\zeta$ be
          the weighted nondeterministic Turing machines for $\psi$ and $\zeta$, respectively.
          We construct a Turing machine $\mathcal{M}$ for $\phi$ which
          first simulates $\mathcal{M}_\psi$ and then $\mathcal{M}_\zeta$ on the input.
          All transitions outside of the simulations,
          e.g., preparing the input for a simulation and clearing the tape for the second simulation,
          have weight $1$ and are deterministic.
          Thus, every combination of a run $r_\psi$ of $\mathcal{M}_\psi$
          and a run $r_\zeta$ of $\mathcal{M}_\zeta$ on the input
          is simulated by exactly one run of $\mathcal{M}$
          and the weight of this run is the product of the
          weights of $r_\psi$ and $r_\zeta$.
          As multiplication distributes over addition,
          $\mathcal{M}$ recognizes $\phi$.
    \item If $\phi = \bigoplus x.\psi$,
          we let $\mathcal{M}_\psi$ be
          the weighted nondeterministic Turing machine for $\psi$.
          We construct a Turing machine $\mathcal{M}$ for $\phi$ which nondeterministically
          guesses an element $a \in A$ and simulates
          $\mathcal{M}_\psi$ on the input $\tuple{\mathfrak{A}, a, a_1, \dotsc, a_l}$.
          We ensure that for every $a \in A$, there is exactly one run of $\mathcal{M}$
          which guesses $a$ and that all transitions which prepare the simulation
          have weight $1$.
          Thus, for every choice of $a \in A$,
          every run of $\mathcal{M}_\psi$ on $\tuple{\mathfrak{A}, a, a_1, \dotsc, a_l}$
          is simulated by exactly one run of $\mathcal{M}$.
    \item If $\phi = \bigfmal x.\psi$,
          we let $\mathcal{M}_\psi$ be
          the weighted nondeterministic Turing machine for $\psi$.
          We construct a Turing machine for $\phi$ which iterates
          over all elements $a \in A$ and simulates
          $\mathcal{M}_\psi$ on the input $\tuple{\mathfrak{A}, a, a_1, \dotsc, a_l}$.
          All transitions outside of the simulations,
          e.g., preparing the input for a simulation and clearing the tape for the next simulation,
          have weight $1$ and are deterministic.
          Thus, every combination of runs $r_0, \dotsc, r_{n-1}$ of $\mathcal{M}_\psi$
          on $\tuple{\mathfrak{A}, 0, a_1, \dotsc, a_l}, \dotsc, \tuple{\mathfrak{A}, n-1, a_1, \dotsc, a_l}$, respectively,
          is simulated by exactly one run of $\mathcal{M}$
          and the weight of this run is the product of the
          weights of $r_0, \dotsc, r_{n-1}$.
          As multiplication distributes over addition,
          $\mathcal{M}$ recognizes $\phi$.
    \item If $\phi = \bigoplus X.\psi$,
          we proceed like in the case of $\bigoplus x.\psi$
          but guess a relation $R$ of appropriate arity for $X$
          instead of an element $a \in A$.
  \end{itemize}

Now, for (2) from Definition~\ref{ch}, suppose that $P$ is recognizable by a non-deterministic weighted Turing machine in polynomial time. Let $\mathcal{M}= \tuple{Q, \Gamma,\Delta, \nu, q_0, F, \Box}$ be the machine  such that $\|\mathcal{M}\| = P$ in time $n^k$ where $n$ is the length of the input (the encoding of a structure from  $\str{\tau}_<$) and $k$ is bigger than the maximum of the arities of the relational symbols in $\tau$. We will construct a $\mathrm{wESO}$-formula $\phi$ such that  $\|\phi\|= P$. Let $Q=\{q_0, \dots, q_{m-1}\}$.  If $\tau=\emptyset$, we can set $\text{enc}(\mathfrak{A})=\underbrace{0\dotsb 0}_{\mathclap{|{A}|-\text{times}}}$ by definition, to make sure that $\text{enc}(\mathfrak{A})$ is at least the same length as $|{A}|$.
Thus, we may take $\Gamma$ in our Turing machine to be $\{0,1\}$.

The first task is to build a Boolean second-order formula $$\psi(T_0, T_1, T_2, H_{q_0}, \dots, H_{q_{m-1}}),$$  without second-order quantifiers but where $T_0, T_1, T_2, H_{q_0}, \dots, H_{q_{m-1}}$ are new second-order variables, such that there is a one-to-one correspondence between the accepting computation paths for the input $\text{enc}(\mathfrak{A})$ and the expansions of the model $\mathfrak{A}$ that satisfy  $\psi(T_0, T_1, T_2, H_{q_0}, \dots, H_{q_{m-1}})$, i.e.\ where $\psi(T_0, T_1, T_2, H_{q_0}, \dots, H_{q_{m-1}})$ takes value $1$.

Recall that $\mathfrak{A}$ with $|A|=n$ is linearly ordered by $<$. We will represent the $n^k$ time and space  parameters as the elements of the set $A^k$, so $k$-tuples from $A$. From $<$, we can define in first-order logic an associated successor relation $\mathtt{Succ}$, as well as the bottom $\bot$ and top elements $\top$. With this at hand, if $\overline{x}, \overline{y}$ are $k$-tuples of variables, we can define a successor relation on the elements of $A^k$ by the formula 
$$\overline{x} = \overline{y} +1 := \bigwedge_{i<k} (\bigwedge_{j<i} (x_j=\top \wedge y_j=\bot) \wedge \mathtt{Succ(x_i,y_i) \wedge } \bigwedge_{j>i} x_j=y_j).$$

Next, let us spell out the meaning of the predicates $T_0, T_1, T_2, H_{q_0}, \dots, H_{q_{m-1}}$:

\begin{itemize}
\item[(1)] $T_i (\overline{p}, \overline{t}) (i=0,1)$, where  $\overline{p}, \overline{t}$ are $k$-tuples of first-order variables, is meant to represent that at time $\overline{t}$ the position $\overline{p}$ of the tape contains the symbol $i$. $T_2$ does the same but for the blank symbol.

\item[(2)] $H_{q_i} (\overline{p}, \overline{t}) (0\leq i \leq m-1)$   represents that at time $\overline{t}$ the machine is in state $q_i$ with its head in position $\overline{p}$.
\end{itemize}

The idea is that we will use the predicates  $T_i$'s and $H_q$'s to describe an accepting computation of $\mathcal{M}$ started with input $\text{enc}(\mathfrak{A})$.
For simplicity, we would like to assume that all computations on a structure of size $n$ have length $n^k$,
i.e., that there are no shorter computations.
In order to do this,
we modify the Turing machine by adding transitions which ``do nothing'' from all accepting states.
Note that our definition of Turing machines specifically forbids transitions from accepting states.
The resulting Turing machine has the same behavior as our given Turing machine when only taking into account
computations of exactly length $n^k$, which is what the formula we construct will do.
More precisely, we add the transitions
$\{\tuple{q,a,q,a,0} \mid q \in F, a \in \Gamma\}$
and assign weight $1$ to all of them.

Given $k$-tuples of variables $\overline{x}=x_1, \dots, x_k$ and $\overline{y}=y_1, \dots, y_k$, we write $\overline{x}\neq \overline{y}$ as an abbreviation for $\bigvee_{1\leq i\leq k}x_i\neq y_i$.
We let $\psi(T_0, T_1, T_2, H_{q_0}, \dots, H_{q_{m-1}})$ be the conjunction of the following:
\begin{itemize}
\item 
$\forall\overline{p}\forall\overline{t} (T_0(\overline{p}, \overline{t}) \leftrightarrow \neg T_1(\overline{p}, \overline{t}) )$


``In every configuration no cell of the tape contains more than one symbol from the alphabet $\Gamma$.'' 

\item  $\forall\overline{t}\exists ! \overline{p} (\bigwedge_{q \in Q} H_q (\overline{p}, \overline{t})) \wedge \forall \overline{p}\forall\overline{t} (\bigwedge_{\substack{q,q'\in Q  q\neq q'}} (\neg H_q(\overline{p}, \overline{t}) \vee \neg H_{q'}(\overline{p}, \overline{t}))) $


``At any time the machine $\mathcal{M}$ is in exactly one state.'' 

\item $\exists\overline{t}\exists  \overline{p} (\bigvee_{q\in F} H_q(\overline{p}, \overline{t}))$


``Eventually the machine $\mathcal{M}$ enters an accepting state.''

\item $\bigvee_{\substack{\tuple{p,a,q,b,D} \in \Delta D \in \{-1,0,1\}}} \theta_{\tuple{p,a,q,b,D}}$, where 

$\theta_{\tuple{p,a,q,b,-1}}:= \forall \overline{t}\forall  \overline{p}((H_p( \overline{p},  \overline{t}) \wedge T_a( \overline{p},  \overline{t})) \rightarrow(H_q(\overline{p}-1,\overline{t}+1) \wedge T_b(\overline{p},\overline{t}+1)) \wedge \forall \overline{p}'(\overline{p}\neq \overline{p}' \rightarrow (\bigwedge_{i=0,1,2} (T_i(\overline{p}', \overline{t}+1) \leftrightarrow T_i(\overline{p}', \overline{t})))) )
$

$\theta_{\tuple{p,a,q,b,1}}:= \forall \overline{t}\forall  \overline{p}((H_p( \overline{p},  \overline{t}) \wedge T_a( \overline{p},  \overline{t})) \rightarrow(H_q(\overline{p}+1,\overline{t}+1) \wedge T_b(\overline{p},\overline{t}+1)) \wedge \forall \overline{p}'(\overline{p}\neq \overline{p}' \rightarrow (\bigwedge_{i=0,1,2} (T_i(\overline{p}', \overline{t}+1) \leftrightarrow T_i(\overline{p}', \overline{t})))) )
$

$\theta_{\tuple{p,a,q,b,0}}:= \forall \overline{t}\forall  \overline{p}((H_p( \overline{p},  \overline{t}) \wedge T_a( \overline{p},  \overline{t})) \rightarrow(H_q(\overline{p},\overline{t}+1) \wedge T_b(\overline{p},\overline{t}+1)) \wedge \forall \overline{p}'(\overline{p}\neq \overline{p}' \rightarrow (\bigwedge_{i=0,1,2} (T_i(\overline{p}', \overline{t}+1) \leftrightarrow T_i(\overline{p}', \overline{t})))) )
$


``The configurations respect the transitions in $\Delta$.'' 

\item $H_{q_0}( \underbrace{\bot \cdot\dots \cdot \bot}_{k-\text{times}},  \underbrace{\bot \cdot\dots \cdot \bot}_{k-\text{times}}) \wedge \bigwedge_{R_i\in \tau}\forall x_1, \dots ,x_{r_i} (( R_i (x_1, \dots ,x_{r_i})\rightarrow \\ T_1( \underbrace{\bot \cdot\dots \cdot \bot}_{k-r_i-\text{times}}x_1 \dots x_{r_i}, \underbrace{\bot \cdot\dots \cdot \bot}_{k-\text{times}}) )\wedge ( \neg R_i (x_1, \dots ,x_{r_i})\rightarrow T_0( \underbrace{\bot \cdot\dots \cdot \bot}_{k-r_i-\text{times}}x_1 \dots x_{r_i}, \underbrace{\bot \cdot\dots \cdot \bot}_{k-\text{times}})) \wedge  \forall x_1, \dots ,x_{k} ((x_1\neq \bot \vee \cdots \vee x_{k-r_i}\neq \bot)\rightarrow T_2( x_1\cdots x_k, \underbrace{\bot \cdot\dots \cdot \bot}_{k-\text{times}}))$

``At the initial time the tape contains $\text{enc}(\mathfrak{A})$ and it is in the initial state $q_0$.'' 

Here $r_i$ is the arity of the relation symbol $R_i$.
\end{itemize}

Finally, to get a $\mathrm{wESO}$-formula $\phi$ such that  $\| \phi\|= P$, we must first consider $\chi$:
$$\psi(T_0, T_1, T_2, H_{q_0}, \dots, H_{q_{m-1}})\fmal
\bigfmal {t_1} \ldots \bigfmal {t_k} \bigfplus_{\substack{\tuple{p,a,q,b,D} \in \Delta D \in \{-1,0,1\}}} \wt\tuple{p,a,q,b,D}
\fmal \beta_{\tuple{p,a,q,b,D}}(\overline{t})$$
where 
\noindent$\beta_{\tuple{p,a,q,b,1}}(\overline{t}) :=
\exists p_1 \ldots \exists p_k \exists q_1 \ldots \exists q_k \exists s_1 \ldots \exists s_k
\Big( (\overline{p} = \overline{q} + 1) \land (\overline{s} = \overline{t} + 1 ) \land H_p(\overline{p},\overline{t}) \land T_a(\overline{p},\overline{t}) \land H_q(\overline{q}, \overline{s}) \land T_b(\overline{p}, \overline{s}) \Big),
$

\noindent$\beta_{\tuple{p,a,q,b,-1}}(\overline{t}) :=
\exists p_1 \ldots \exists p_k \exists q_1 \ldots \exists q_k \exists s_1 \ldots \exists s_k
\Big( (\overline{p} = \overline{q} - 1) \land (\overline{s} = \overline{t} + 1 ) \land H_p(\overline{p},\overline{t}) \land T_a(\overline{p},\overline{t}) \land H_q(\overline{q}, \overline{s}) \land T_b(\overline{p}, \overline{s}) \Big),$

\noindent$\beta_{\tuple{p,a,q,b,0}}(\overline{t}) :=
\exists p_1 \ldots \exists p_k \exists q_1 \ldots \exists q_k \exists s_1 \ldots \exists s_k
\Big( (\overline{p} = \overline{q}) \land (\overline{s} = \overline{t} + 1 ) \land H_p(\overline{p},\overline{t}) \land T_a(\overline{p},\overline{t}) \land H_q(\overline{q}, \overline{s}) \land T_b(\overline{p}, \overline{s}) \Big),$ which intuitively tells us that $\tuple{p,a,q,b,D}$ is a transition made by some configuration in the computation in question.
The point of $\chi$ is that when we get for 
some interpretations of $T_0, T_1, T_2, H_{q_0},\dots, H_{q_{m-1}}$ an accepting computation $C_1e_1C_2e_2C_3\ldots C_ne_nC_{n+1}$, we want $\chi$ to give us its weight, $\nu(e_1)\nu(e_2)\ldots \nu(e_n)$ as value. The order of the tuples $\tuple{t_1, \ldots, t_k}$ in the universal quantification in $\chi$ reflects their  enumeration in the lexicographic order (indeed, our quantifiers respect the order of the structure and the evaluation of the quantifiers starts with the innermost and ends with the outermost quantifier). With all this in mind, $\phi$ is $\bigoplus T_0 \bigoplus T_1 \bigoplus T_2 \bigoplus H_{q_0} \ldots \bigoplus H_{q_{m-1}} \chi.$

$(ii)$: For this part we reason similarly as before. In the  proof of (1) we must observe that the semantics of multiplicative quantifiers can be now defined independently of the order thanks to the commutativity of the multiplication. In the proof of (2) we simply consider a Boolean formula $\theta (L)$ (which takes as only possible values $0$ or $1$) that expresses that the binary relation $L$ is a suitable ordering, we take the formula 
$\bigoplus L \bigoplus T_0 \bigoplus T_1 \bigoplus T_2 \bigoplus H_{q_0} \ldots \bigoplus H_{q_{m-1}} (\theta (L) \fmal \chi)$, and we replace every formula $x \leq y$ by $L(x,y)$. Note that in ($i$), if multiplication is commutative, the value of $\phi$ is independent of and the same for any given order. Thus, idempotency of  $\oplus$ guarantees that the value of $\bigoplus T_0 \bigoplus T_1 \bigoplus T_2 \bigoplus H_{q_0} \ldots \bigoplus H_{q_{m-1}} (\theta (L) \fmal \chi)$ is going to be that of $\bigoplus T_0 \bigoplus T_1 \bigoplus T_2 \bigoplus H_{q_0} \ldots \bigoplus H_{q_{m-1}}  \chi$ when $L$ is indeed an ordering.
\end{proof}

\begin{Rmk} Notice that in part (1) of the proof of Theorem~\ref{thm:wfagin} the constructed weighted Turing machine uses as weights for the transitions, besides $0$ and $1$, the same weights that occur in the given formula $\phi$. Analogously, in  part (2) the constructed formula uses only the weights that appeared in the transitions of the given Turing machine. Moreover, the two constructions are effective for all semirings.
\end{Rmk}

\setcounter{Thm}{14}

\begin{Cor}[Weighted Cook--Levin's theorem]
Let $\mathcal{S}$ be a finitely generated semiring. Then, $\mathsf{SAT}[\mathcal{S}]$ is $\mathsf{NP}[\mathcal{S}]$-complete.
\end{Cor}

\begin{proof}
We know that $\mathsf{SAT}[\mathcal{S}]$ is in $\mathsf{NP}[\mathcal{S}]$, so all that is left to show is that any series $\sigma\colon \Sigma^* \longrightarrow S$ (where $\Sigma$ is an alphabet) recognizable in $\mathsf{NP}[\mathcal{S}]$ is polynomially many-one reducible to the series $\text{SAT}[\mathcal{S}]$. First observe that the set of words $\Sigma^*$ can be regarded as a set $\text{Struct}_<[\tau]$ of ordered finite structures for a vocabulary $\tau$ (namely, the vocabulary that has a unary predicate for each symbol of the alphabet). Thus, by the weighted Fagin's theorem, we have a $\mathrm{wESO}$-formula $\phi$ such that $\| \phi\|= \sigma$. Our goal consists in finding a weighted propositional formula $\phi'$ such that $\|\phi'\|= \| \phi\|$. 

We may assume that $\phi= \bigoplus P_1, \dots P_n \psi$ as described in Fagin's theorem. Next, we polynomially associate any $\mathfrak{A} \in \text{Struct}_<[\tau]$ with a propositional formula $\psi_\mathfrak{A}\in \mathtt{Fmla}[\mathcal{S}]$ such that $\text{SAT}[\mathcal{S}](\psi_\mathfrak{A}) = \|\phi\|(\mathfrak{A}) = \sigma(\mathfrak{A})$. Start by considering a propositional vocabulary $\{P_i^{\overline{a}} \mid i=1,\dots, n, \overline{a} \in A^{\text{ar}(P_i)}\} \cup \{Q^{a=b}, Q^{a<b} \mid a,b \in A\}$ and suppose that we have first-order constants $\{a \mid a \in A\}$.  Replace every  quantifier $\exists x \theta(x)$  by the formula $\bigvee_{a \in A} \theta(x/a)$. Then, replace every quantifier of the form $\forall x \theta(x)$  by the formula $\bigwedge_{a \in A} \theta(x/a)$.
Then, replace every formula of the form $a<b, a=b$ or $R(\overline{a})$ ($R \in \tau$) by its corresponding truth-value in $\mathfrak{A}$ (i.e.\ $0$ or $1$). Finally, replace every formula of the form $P_i(\overline{a})$ by the propositional variable $P_i^{\overline{a}}$. The resulting  propositional formula $\psi_\mathfrak{A}$ is such that  $$\text{SAT}[\mathcal{S}](\psi_\mathfrak{A}) = \sum_{W \in \{0,1\}^{X_{\psi_\mathfrak{A}}}} \overline{W}(\psi_\mathfrak{A}) = \sum_{\substack{I_i \subseteq A^{\text{ar}(P_i)} 
i=1, \dots, n}} \| \psi\|(\mathfrak{A},  I_1, \dots, I_1) = \| \phi\|(\mathfrak{A}).$$
\end{proof}

Now we need some notation that we will use in the next few proofs. For any two formulas $\beta $ and $\varphi$,
we  define the abbreviation $\beta \triangleright \varphi = (\beta \fmal \varphi) \fplus (\lnot \beta \fmal \sone)$, i.e.,
\[\ser{\beta \triangleright \varphi}(\mfa, \rho) = \begin{cases} \ser{\varphi}(\mfa, \rho) & \text{if } \tuple{\mfa, \rho} \models \beta \\ \sone & \text{otherwise.} \end{cases}\]

\begin{Thm}[Weighted Immerman--Vardi's theorem] The logic $\mathrm{wLFP}[\mathcal{S}]$ (with weights in a semiring $\mathcal{S}$) \emph{captures} $\mathsf{FP}[\mathcal{S}]$ over ordered structures in the vocabulary $\tau=\{R_1, \dots, R_j\}$. 

\end{Thm}

\begin{proof}
To show (1) from Definition~\ref{ch}, first note that every $\mathrm{LFP}$-formula $\beta$ can be evaluated in polynomial time and hence a polynomial time Turing machine can output $1$ or $0$ depending on whether $\beta$ is satisfied or not.
Also, for a semiring element $s$, the Turing machine outputting the term $s$ for every input runs in constant and hence polynomial time.

Furthermore, $\mathsf{FP}[\mathcal{S}]$ is closed under polynomial sums as we may compute a term of polynomially many summands, each of which is computable in polynomial time, in polynomial time. Similarly, $\mathsf{FP}[\mathcal{S}]$ is closed under polynomial products.

To show (2), suppose that $\sigma \in \mathsf{FP}[\mathcal{S}]$, that is, $\sigma \colon \Sigma^* \longrightarrow \langle G \rangle$ for a finite $G\subseteq S$ and a finite alphabet $\Sigma$, and there exists a polynomial-time deterministic Turing machine which given a word $\text{enc}(\mathfrak{A})$ outputs a word $w_{\text{enc}(\mathfrak{A})}$ in the algebra of terms $\mathcal{T}(G)$ such that $w_{\text{enc}(\mathfrak{A})}$ evaluates to $\sigma(\text{enc}(\mathfrak{A}))$.
Then for some $l \in \mathbb{N}$, we have $|w_{\text{enc}(\mathfrak{A})}| \leq |A|^l$ for all structures $\mathfrak{A}$, where $A$ is the universe of $\mathfrak{A}$.
Like in the proof of Theorem~\ref{thm:wfagin},
we encode numbers in $\{0,\dotsc,|A|^l-1\}$ using tuples from $A^l$.

For each $s_p \in G=\{s_1, \dots, s_\ell\}$, consider the language \[\mathcal{L}_p = \{ \tuple{\mathfrak{A}, \overline{a}_1, \overline{b}_1, \dotsc, \overline{a}_k, \overline{b}_k} \mid w_{\text{enc}(\mathfrak{A})} = \bigsplus^{m_1}_{i_1=1} \bigsmal^{n_1}_{j_1=1} \dotsb \bigsplus^{m_k}_{i_k=1} \bigsmal^{n_k}_{j_k=1} s_{i_1j_1 \dotsb i_k j_k} \ \text{and} \  s_{ \overline{a}_1  \overline{b}_1 \dotsb \overline{a}_k \overline{b}_k} = s_p\}.\]
Note that $m_1,n_1,\dotsc,m_k,n_k$ are all bounded by $A^l$.
Then $\mathcal{L}_p$ is recognizable in polynomial time, so by~\cite{Immerman1, Vardi}, there is an $\mathrm{IFL}$-formula $\phi_p(\overline{x}_1, \overline{y}_1,\dotsc, \overline{x}_k, \overline{y}_k)$ such that $\mathfrak{A} \models \phi_p(\overline{a}_1, \overline{b}_1, \dotsc, \overline{a}_k, \overline{b}_k)$ iff $\tuple{\mathfrak{A}, \overline{a}_1, \overline{b}_1, \dotsc, \overline{a}_k, \overline{b}_k} \in \mathcal{L}_p$.

Now we take the $\mathrm{wIFL}$-formula $\psi:= \bigfplusop \overline{x}_1 \bigfmalop \overline{y}_1 \dotsb \bigfplusop \overline{x}_k \bigfmalop \overline{y}_k \bigfmal_{p=1}^\ell(\phi_p(\overline{x}_1, \overline{y}_1, \dotsc, \overline{x}_k, \overline{y}_k) \triangleright s_p)$.
We have then that $\|\psi\|(\mathfrak{A})$ is exactly $\bigsplus^{m_1}_{i_1=1} \bigsmal^{n_1}_{j_1=1} \dotsb \bigsplus^{m_k}_{i_k=1} \bigsmal^{n_k}_{j_k=1} s_{i_1j_1 \dotsb i_k j_k}= \sigma(\text{enc}(\mathfrak{A}))$.
\end{proof}

\setcounter{Thm}{18}
\begin{Thm}
The logic $\mathrm{wPFP}[\mathcal{S}] + \{\prod X, \sum X\}$ (with weights in a semiring $\mathcal{S}$) \emph{captures} $\mathsf{FPSPACE}[\mathcal{S}]$ over ordered structures in the vocabulary $\tau=\{R_1, \dots, R_j\}$. 
\end{Thm}

\begin{proof}
To show (1) from Definition~\ref{ch}, we proceed like in Theorem~\ref{I-V}.
First, every $\mathrm{PFP}$-formula $\beta$ can be evaluated in $\mathsf{PSPACE}$, so we may compute its characteristic function in $\mathsf{PSPACE}$ as well. Also, constant functions can be computed in $\mathsf{PSPACE}$.
Finally, $\mathsf{FPSPACE}[\mathcal{S}]$ is closed under exponential sums since an exponential counter can be stored in polynomial space. Similarly, $\mathsf{FPSPACE}[\mathcal{S}]$ is closed under exponential products.

To show (2), suppose that $\sigma \in \mathsf{FPSPACE}[\mathcal{S}]$, that is, $\sigma \colon \Sigma^* \longrightarrow \langle G \rangle$ for a finite $G\subseteq S$ and a finite alphabet $\sigma$, and there exists a polynomial-space deterministic Turing machine which given a word $\text{enc}(\mathfrak{A})$ outputs a word $w_{\text{enc}(\mathfrak{A})}$ in the algebra of terms $\mathcal{T}(G)$ such that $w_{\text{enc}(\mathfrak{A})}$ evaluates to $\sigma(\text{enc}(\mathfrak{A}))$.
Then, for some $l \in \mathbb{N}$, we have $|w_{\text{enc}(\mathfrak{A})}| \leq 2^{|A|^l}$ for all structures $\mathfrak{A}$, where $A$ is the universe of $\mathfrak{A}$.
We encode numbers in $\{0,\dotsc,2^{|A|^l}-1\}$ using subsets of $A^l$ as follows.

Let $\pi\colon 2^{A^l} \longrightarrow \mathbb{N}$ be a function that lets $\pi(B)$ be the number of relations in  $2^{A^l}$ that are smaller than  $B$ according to the following induced linear order on relations of arity $l$: $X <^* Y$ iff there is $u\in Y \setminus X$ such that if $v> u$, $v \in Y$ iff $u \in X$. For each $s_p \in G=\{s_1, \dots, s_\ell\}$, consider the language \[\mathcal{L}_p = \{ \tuple{\mathfrak{A}, B_1, C_1, \dotsc, B_k, C_k} \mid w_{\text{enc}(\mathfrak{A})} = \bigsplus^{m_1}_{i_1=1} \bigsmal^{n_1}_{j_1=1} \dotsb \bigsplus^{m_k}_{i_k=1} \bigsmal^{n_k}_{j_k=1\mathcal{L}} s_{i_1j_1 \dotsb i_k j_k},  \  \] \[\ \,\,\,\, \,\,\,\, \,\,\,\, \,\,\,\, B_1, C_1, \dotsc, B_k, C_k \subseteq A^l \, \, \text{and} \, \,  s_{ \pi(B_1) \pi(C_1) \dotsb \pi(B_k) \pi(C_k) } =   s_p\}.\]
Note that $m_1,n_1,\dotsc,m_k,n_k$ are all bounded by $2^{A^l}$.
Then, $\mathcal{L}_p$ is recognizable in $\mathsf{PSPACE}$ as $s_{ \pi(B_1) \pi(C_1) \dotsb \pi(B_k) \pi(C_k) }$ can be computed in $\mathsf{PSPACE}$ and compared to $s_p$, so by~\cite{Vianu, Vardi}, there is a $\mathrm{PFP}$-formula $\phi_p(X_1, Y_1, \dotsc, X_k, Y_k)$ such that\\ $\tuple{\mathfrak{A}, B_1, C_1, \dotsc, B_k, C_k} \models \phi_p(X_1, Y_1, \dotsc, X_k, Y_k)$ iff $\tuple{\mathfrak{A}, B_1, C_1, \dotsc, B_k, C_k} \in \mathcal{L}_p$.

Now we take the $\mathrm{wPFP}[\mathcal{S}]+ \{\prod X, \sum X\}$-formula\\ $\psi:= \bigfplusop X_1 \bigfmalop Y_1 \dotsb \bigfplusop X_k \bigfmalop Y_k \bigfmal_{p=1}^\ell(\phi_p(X_1, Y_1, \dotsc, X_k, Y_k) \triangleright s_p)$.
%
%
We have then that $\|\psi\|(\mathfrak{A})$ is exactly $\bigsplus^{m_1}_{i_1=1} \bigsmal^{n_1}_{j_1=1} \dotsb \bigsplus^{m_k}_{i_k=1} \bigsmal^{n_k}_{j_k=1} s_{i_1j_1 \dotsb i_k j_k}= \sigma(\text{enc}(\mathfrak{A}))$.

\end{proof}

\setcounter{Thm}{20}
\begin{Thm}
The logic $\mathrm{wPFP}[\mathcal{S}]$ (with weights in a semiring $\mathcal{S}$) \emph{captures} $\mathsf{FPSPACE}_{poly}[\mathcal{S}]$ over ordered structures in the vocabulary $\tau=\{R_1, \dots, R_j\}$. 
\end{Thm}

\begin{proof} To show (1) from Definition~\ref{ch}, we need to prove that $\mathsf{FPSPACE}_{poly}[\mathcal{S}]$ is closed under the relevant semiring operations. We proceed as in the first half of Theorem~\ref{I-V}.

To show (2), suppose that $\sigma \in \mathsf{FPSPACE}_{poly}[\mathcal{S}]$, that is, $\sigma \colon \Sigma^* \longrightarrow \langle G \rangle$ for a finite $G\subseteq S$ and a finite alphabet $\Sigma$, and there exists a polynomial-space deterministic Turing machine with polynomial size output which given a word $\text{enc}(\mathfrak{A})$ outputs a word $w_{\text{enc}(\mathfrak{A})}$ in the algebra of terms $\mathcal{T}(G)$ such that $w_{\text{enc}(\mathfrak{A})}$ evaluates to $\sigma(\text{enc}(\mathfrak{A}))$.
Then similar to the proof of Theorem~\ref{I-V},
there exists some $l \in \mathbb{N}$ with $|w_{\text{enc}(\mathfrak{A})}| \leq |A|^l$ for all structures $\mathfrak{A}$, where $A$ is the universe of $\mathfrak{A}$.

For each $s_p \in G=\{s_1, \dots, s_\ell\}$, consider the language \[\mathcal{L}_p = \{ \tuple{\mathfrak{A}, \overline{a}_1, \overline{b}_1, \dotsc, \overline{a}_k, \overline{b}_k} \mid w_{\text{enc}(\mathfrak{A})} = \bigsplus^{m_1}_{i_1=1} \bigsmal^{n_1}_{j_1=1} \dotsb \bigsplus^{m_k}_{i_k=1} \bigsmal^{n_k}_{j_k=1} s_{i_1j_1 \dotsb i_k j_k} \ \text{and} \  s_{ \overline{a}_1  \overline{b}_1 \dotsb \overline{a}_k \overline{b}_k} =   s_p\}.\]
$\mathcal{L}_p$ is recognizable in
 $\mathsf{PSPACE}$, so by~\cite{Vianu, Vardi}, there is a $\mathrm{PFP}$-formula 
  $\phi_p(\overline{x}_1, \overline{y}_1, \dotsc, \overline{x}_k, \overline{y}_k)$ such that $\mathfrak{A} \models \phi_p(\overline{a}_1, \overline{b}_1, \dotsc, \overline{a}_k, \overline{b}_k)$ iff $(\mathfrak{A}, \overline{a}_1, \overline{b}_1, \dotsc, \overline{a}_k, \overline{b}_k) \in \mathcal{L}_p$.

Now we take the $\mathrm{wPFP}$-formula $\psi:= \bigfplusop \overline{x}_1 \bigfmalop \overline{y}_1 \dotsb \bigfplusop \overline{x}_k \bigfmalop \overline{y}_k \bigfmal_{p=1}^\ell(\phi_p(\overline{x}_1, \overline{y}_1, \dotsc, \overline{x}_k, \overline{y}_k) \triangleright s_p)$.
%
%
We have then that $\|\psi\|(\mathfrak{A})$ is exactly $\bigsplus^{m_1}_{i_1=1} \bigsmal^{n_1}_{j_1=1} \dotsb \bigsplus^{m_k}_{i_k=1} \bigsmal^{n_k}_{j_k=1} s_{i_1j_1 \dotsb i_k j_k}= \sigma(\text{enc}(\mathfrak{A}))$.
\end{proof}

\setcounter{Thm}{22}
\begin{Thm}
The logic $\mathrm{wDTC}[\mathcal{S}]$ (with weights in a semiring $\mathcal{S}$) \emph{captures} $\mathsf{FPLOG}[\mathcal{S}]$ over ordered structures in the vocabulary $\tau=\{R_1, \dots, R_j\}$. 
\end{Thm}

\begin{proof} To show (1) from Definition~\ref{ch}, again we need to prove that $\mathsf{FPLOG}[\mathcal{S}]$ is closed under the relevant semiring operations. We proceed as in the first half of Theorem~\ref{I-V} and note that $\mathsf{FPLOG}[\mathcal{S}]$ is closed under polynomial sums and products since polynomial counters can be stored in logarithmic space.

To show (2), suppose that $\sigma \in \mathsf{FPLOG}[\mathcal{S}]$, that is, $\sigma \colon \Sigma^* \longrightarrow \langle G \rangle$ for a finite $G\subseteq S$ and a finite alphabet $\Sigma$, and there exists a logarithmic-space deterministic Turing machine such that given a word $\text{enc}(\mathfrak{A})$ outputs a word $w_{\text{enc}(\mathfrak{A})}$ in the algebra of terms $\mathcal{T}(G)$ such that $w_{\text{enc}(\mathfrak{A})}$ evaluates to $\sigma(\text{enc}(\mathfrak{A}))$.
Again, since the output size of a logarithmic-space Turing machine is at most polynomial,
there exists some $l \in \mathbb{N}$ with $|w_{\text{enc}(\mathfrak{A})}| \leq |A|^l$ for all structures $\mathfrak{A}$, where $A$ is the universe of $\mathfrak{A}$.

For each $s_p \in G=\{s_1, \dots, s_\ell\}$, consider the language \[\mathcal{L}_p = \{ \tuple{\mathfrak{A}, \overline{a}_1, \overline{b}_1, \dotsc, \overline{a}_k, \overline{b}_k} \mid w_{\text{enc}(\mathfrak{A})} = \bigsplus^{m_1}_{i_1=1} \bigsmal^{n_1}_{j_1=1} \dotsb \bigsplus^{m_k}_{i_k=1} \bigsmal^{n_k}_{j_k=1} s_{i_1j_1 \dotsb i_k j_k} \ \text{and} \  s_{ \overline{a}_1  \overline{b}_1 \dotsb \overline{a}_k \overline{b}_k} =   s_p\}.\]
$\mathcal{L}_p$ is recognizable in
 $\mathsf{DLOGSPACE}$, so by~\cite{Immerman1}, there is a $\mathrm{DTC}$-formula 
  $\phi_p(\overline{x}_1, \overline{y}_1, \dotsc, \overline{x}_k, \overline{y}_k)$ such that $\mathfrak{A} \models \phi_p(\overline{a}_1, \overline{b}_1, \dotsc, \overline{a}_k, \overline{b}_k)$ iff $\tuple{\mathfrak{A}, \overline{a}_1, \overline{b}_1, \dotsc, \overline{a}_k, \overline{b}_k} \in \mathcal{L}_p$.

Now we take the $\mathrm{wTC}$-formula $\psi:= \bigfplusop \overline{x}_1 \bigfmalop \overline{y}_1 \dotsb \bigfplusop \overline{x}_k \bigfmalop \overline{y}_k \bigfmal_{p=1}^\ell(\phi_p(\overline{x}_1, \overline{y}_1, \dotsc, \overline{x}_k, \overline{y}_k) \triangleright s_p)$.
%
%
We have then that $\|\psi\|(\mathfrak{A})$ is exactly $\bigsplus^{m_1}_{i_1=1} \bigsmal^{n_1}_{j_1=1} \dotsb \bigsplus^{m_k}_{i_k=1} \bigsmal^{n_k}_{j_k=1} s_{i_1j_1 \dotsb i_k j_k}= \sigma(\text{enc}(\mathfrak{A}))$.
\end{proof}

\setcounter{Thm}{23}
\begin{Rmk}
There appears to be a pattern behind the proofs of the preceding theorems. However, it is not obvious whether the theorems and their proofs can be fit into a common framework as the appropriate weighted quantifiers for the logical characterization are specific to the complexity class. For instance, $\mathsf{FP}[\mathcal{S}]$ requires polynomial sums and products and is also closed under them, $\mathsf{FPSPACE}[\mathcal{S}]$ on the other hand requires exponential sums and products.
Other complexity classes may require additional restrictions to the quantifiers and closure properties under sums or products of a certain size may require arguments specific to the class.
\end{Rmk}

\setcounter{Thm}{24}
\begin{Pro} Let $\mathcal{R}$ be a commutative semiring. There is a series $P\in \text{\emph{NP}}[\mathcal{R}]$ such that for no $\varphi \in \mathrm{wESO}$, $|| \varphi|| = P$.

\end{Pro}
\begin{proof}By the weighted Fagin's theorem, it suffices to find $P\in \text{NP}[\mathcal{R}]$ such that $P\notin \mathsf{NP}[\mathcal{R}]$. 
    Let $S$ be any commutative non finitely generated semiring (e.g. the field of rational numbers) and
$\mathcal{M} = (S, \emptyset, \{\iota, \lambda\}, \{\sqcup\}, \iota, \sqcup, \delta)$ a semiring Turing machine,
    where $\delta = \{ (\iota, s, \lambda, s, 1, s) \mid s \in S\}$.
    Then the behavior of $\mathcal{M}$ cannot be modeled by any weighted Turing machine,
    as the set of weights assigned to inputs by a weighted Turing machine are always
    contained in some finitely generated subsemiring of $S$.
\end{proof}

\begin{Pro} Let $\mathcal{R}$ be a commutative semiring and allow only finitely many transitions in a semiring Turing machine. Then $\text{\emph{NP}}[\mathcal{R}] =\mathsf{NP}[\mathcal{R}]$ , i.e.\ the NP class in the sense of~\cite{Eiter} coincides with the NP class in our sense. 

\end{Pro}
\begin{proof} The inclusion $\mathsf{NP}[\mathcal{R}] \subseteq \text{NP}[\mathcal{R}] $ is not too difficult to see. For every weighted Turing machine $\mathcal{M}= \tuple{Q, \Gamma,\Delta, \nu, q_0, F, \Box}$ over a commutative semiring $\mathcal{R}$,
there exists an SRTM
$\mathcal{M}'= \tuple{\mathcal{R}, \mathcal{R}', Q, \Gamma, q_0, \Box, \delta'}$ with the same behavior as $\mathcal{M}$.
For this, choose $\mathcal{R}'$ as the set of all values assigned by $\nu$
and $\delta' = \{(p,a,q,b,d,s) \mid (p,a,q,b,d) \in \Delta, \nu(p,a,q,b,d) = s, s\in \mathcal{R}\}$.
Note that formally, SRTMs always have to move left or right, but introducing transitions which simulate this behavior using a right and a left move are a simple exercise.
The restrictions imposed on SRTMs are clearly satisfied,
as $\mathcal{R}'$ can neither read nor write semiring values,
all transition weights are possible as $\mathcal{R}'$ contains all of the finitely many transitions weights,
and $\mathcal{M}'$ cannot distinguish between semiring values as it it cannot even read them.

We continue with the inclusion  $\text{NP}[\mathcal{R}] \subseteq \mathsf{NP}[\mathcal{R}] $. Suppose that $P \in \text{NP}[\mathcal{R}]$, i.e., $P: \Sigma^* \longrightarrow R$ is a  series such that there is SRTM  $\mathcal{M}=\tuple{\mathcal{R}, \mathcal{R}', Q, \Gamma, q_0, \Box, \delta'}$ that computes $P$ in polynomial time. This mean that  for any $x \in \Sigma^*$, the value of  $\mathcal{M}$ on the configuration $(\iota, x, 0)$ (where $\iota$ is the initial state and $0$ the position of the head) is $P(x)$. We define a weighted Turing machine $\mathcal{M}'= \tuple{Q, \Gamma,\Delta, \nu, q_0, F, \Box}$ by setting $\Delta = \{(p,a,q,b,d) \mid  \ \text{there is} \ s\in \mathcal{R} \ \text{s.t.}  \ (p,a,q,b,d,s) \in \delta'\}$ and $\nu(p,a,q,b,d) = \sum_{(p,a,q,b,d,s) \in \delta'} s$ (this is finite since there are only finitely many transitions in $\delta'$). Observe that in the definition of a SMRTM the same transition can be done with different weights, which is why in our weighted version we need to define this sum. Using distributivity of the semiring, then the function computed by $\mathcal{M}$ (i.e. a series) coincides with the behaviour of $\mathcal{M}'$.\end{proof}


\begin{thebibliography}{30}

\bibitem{Vianu}
Serge Abiteboul and Victor Vianu. Fixpoint extensions of first-order logic and datalog-like languages, \emph{Proceedings of Fourth Annual Symposium on Logic in Computer Science. IEEE Comput. Soc. Press}. pp.\ 71--79, 1989.

\bibitem{Albert-Kari}
Jürgen Albert and Jarkko Kari. Digital Image Compression, in Manfred Droste, Werner Kuich, and Heiko Vogler (editors), \emph{Handbook of Weighted Automata},  Monographs in Theoretical Computer Science, pp.\ 453--479, Springer-Verlag, Berlin, Heidelberg, 2009.
    
\bibitem{Arenas}
Marcelo Arenas, Martin Muñoz, and Cristian Riveros. Descriptive complexity for counting
complexity classes, \emph{Logical Methods in Computer Science} 16(1), 2020.

\bibitem{Be}
Benjam\'in Callejas Bedregal and Santiago Figueira. On the computing power of fuzzy Turing machines, \emph{Fuzzy Sets and Systems} 159(9):1072--1083, 2008.


\bibitem{Bergman}
Bergman, C. (2011). \emph{Universal Algebra: Fundamentals and Selected Topics}. Chapman and Hall/CRC. https://doi.org/10.1201/9781439851302

\bibitem{Berstel-Perrin-Reutenauer}
Jean Berstel, Dominique Perrin, and Christophe Reutenauer. {\em Codes and Automata}, Cambridge University Press, 2009.

\bibitem{Book-Long-Selman}
Ronald V. Book, Timothy J. Long, and Alan L. Selman. Qualitative relativizations of complexity classes, {\em Journal of Computer and System Sciences} 30(3):395--413, 1985.

\bibitem{Cai}
Jin-Yi Cai and Lane A. Hemachandra. On the power of parity, \emph{Proceedings of the 6th Symposium on
Theoretical Aspects of Computer Science}, Lecture Notes in Computer Science,
Vol.\ 349, pp.\ 229--240, Springer-Verlag, Berlin, 1989.

\bibitem{Damm}
Carsten Damm, Markus Holzer, and Pierre McKenzie. The complexity of tensor calculus, \emph{Computational Complexity} 11:54--89, 2002.

\bibitem{Damm1}
Carsten Damm, Markus Holzer, and Pierre McKenzie. The complexity of tensor calculus.
In \emph{Proc. 15th Annual Conference on Computational Complexity Conf.}, IEEE
Comput. Soc. Press, pp.\ 70--86, 2000.

\bibitem{DrosteGastin:07}
Manfred Droste and Paul Gastin. Weighted automata and weighted logics, \emph{Theoretical Computer Science} 380:69--86, 2007.

\bibitem{Droste:Handbook}
Manfred Droste, Werner Kuich, and Heiko Vogler (editors). \emph{Handbook of Weighted Automata},  Monographs in Theoretical Computer Science, Springer-Verlag, Berlin, Heidelberg, 2009.

\bibitem{DrostePaul}
Manfred Droste and Erik Paul.
A Feferman--Vaught decomposition theorem for weighted MSO logic, {\em  43rd International Symposium on Mathematical Foundations of Computer Science (MFCS 2018)} 76:1--15, 2018.

\bibitem{E-F}
Heinz-Dieter Ebbinghaus and Jörg Flum.
\emph{Finite Model Theory}, Perspectives in Mathematical Logic, Springer, 1995.

\bibitem{Eilenberg}
Samuel Eilenberg.
\emph{Automata, Languages, and Machines}, Academic Press, New York and London, 1974.

\bibitem{Eiter}
Thomas Eiter and Rafael Kiesel. Semiring reasoning frameworks in AI and their computational complexity, \emph{Journal of Artificial Intelligence Research} 77:207--293, 2023.

\bibitem{Eiter1}
Thomas Eiter and Rafael Kiesel. On the complexity of sum-of-products problems over semirings. In \emph{AAAI Conference on Artificial
Intelligence}, AAAI-21, pp.\ 6304--6311, 2021.

\bibitem{Fagin}	
Ronald Fagin. Generalized first-order spectra and polynomial-time recognizable sets, \emph{Complexity of computation} 7:43--73, 1974.

\bibitem{Fenner-Fortnow-Kurtz}
Stephen A. Fenner, Lance J. Fortnow, and Stuart A. Kurtz. Gap-definable counting classes, {\em Journal of Computer and System Sciences} 48(1):116--148, 1994.

\bibitem{GM}
Paul Gastin and Benjamin Monmege.
A unifying survey on weighted logics and weighted automata - Core weighted logic: minimal and versatile specification of quantitative properties, {\em Soft Computing} 22(4):1047--1065, 2018.

\bibitem{Gla}
Christian Glasser. Space-efficient informational redundancy, \emph{Journal of Computer and System Sciences} 76(8):792--811, 2010.	

\bibitem{G}	
Erich Grädel. Capturing complexity classes by fragments of second-order logic, \emph{Theoretical Computer
Science} 101(1):35-–57, 1992.

\bibitem{G-K-L}	
Erich Grädel, Phokion Kolaitis, Leonid Libkin, Maarten Marx, Joel Spencer, Moshe Vardi, Yde Venema, and Scott Weinstein. {\em Finite Model Theory and Its Applications}, Springer-Verlag, 2007.


\bibitem{G-S}	
Yuri Gurevich and Saharon Shelah. Fixed-point extensions of first-order logic. \emph{Annals of Pure and Applied Logic} 32:265--280, 1986.

\bibitem{Gupta}
S. Gupta. Closure properties and witness reduction, {\em Journal of Computer and System Sciences} 50(3):412--432, 1995.

\bibitem{Harju}
Tero Harju and Juhani Karhum\"aki. The equivalence problem of multitape finite automata, \emph{Theoretical Computer Science} 78:347--355, 1991.

\bibitem{Immerman1}
Neil Immerman.
Relational queries computable in polynomial time, \emph{Information and Control} 68(1--3):86--104, 1986.

\bibitem{Immerman2}
Neil Immerman.
\emph{Descriptive Complexity}, Graduate texts in computer science, Springer, 1999.

\bibitem{Kos} 
Peter Kostol\'anyi. \emph{Weighted automata and logics meet computational complexity}, arXiv:2312.10810 [cs.FL]. 

\bibitem{Krentel}
Mark W. Krentel. The complexity of optimization problems, {\em Journal of Computer and System Sciences} 36(3):490--509, 1988.

\bibitem{Kuich-Salomaa}
Werner Kuich and Arto Salomaa.
{\em Semirings, Automata, Languages}, Monographs in Theoretical Computer Science, Springer Verlag, 1985.

\bibitem{Ladner}
Richard E. Ladner. Polynomial space counting problems, {\em SIAM Journal on Computing} 18(6):1087--1097, 1989.

\bibitem{Libkin}
Leonid Libkin.
\emph{Elements of Finite Model Theory}, Texts in Theoretical Computer Science. An EATCS Series, Springer, 2004.

\bibitem{Meitus}
V. Yu. Meitus. Decidability of the equivalence problem for
deterministic pushdown automata, \emph{Cybernetics and Systems Analysis} 5:20--45, 1992.

\bibitem{Richerby}
 David Richerby. Logical Characterizations of PSPACE. In Marcinkowski, J., Tarlecki, A. (eds) \emph{Computer Science Logic. CSL 2004}, Lecture Notes in Computer Science, vol.\ 3210, Springer, Berlin, Heidelberg, 2004.

\bibitem{Papadimitriou-Zachos}
Christos H. Papadimitriou and Stathis Zachos. Two remarks on the power of counting. In {\em 6th GI Conferences in Theoretical Computer Science}, pp.\ 269--275, 1983.

\bibitem{Sakarovitch}
Jacques Sakarovitch. {\em Elements of Automata Theory}, Cambridge University Press, 2009.

\bibitem{Salomaa:78}
Arto Salomaa and Matti Soittola.
\emph{Automata-Theoretic Aspects of Formal Power Series}, Monographs in Computer Science, Springer, 1978.

\bibitem{Saluja:95}
Sanjeev Saluja, K. V. Subrahmanyam, and Madhukar N. Thakur. Descriptive complexity of $\#$P
functions. \emph{Journal of Computer and System Sciences} 50(3):493–505, 1995.

\bibitem{Schutzenberger:1961}
Marcel-Paul Sch\"utzenberger. On the definition of a family of automata, \emph{Information and Control} 4(2):245--270, 1961.

\bibitem{Senizergues}
G\'eraud  Senizergues. The equivalence problem for deterministic pushdown
automata is decidable, \emph{Proceedings of International Colloquium on Automata, Languages, and Programming ICALP 1997}, Lecture Notes in Computer Science 1256:671--681, 1997.

\bibitem{Valiant}
Leslie Valiant. The complexity of enumeration and reliability problems. {\em SIAM Journal on Computing} 8(3):410--421, 1979.

\bibitem{Vardi} Moshe Vardi.
The Complexity of Relational Query Languages (Extended Abstract), \emph{STOC 1982 Proceedings of the fourteenth annual ACM symposium on Theory of computing}, pp.\ 137--146, 1982.

\bibitem{Wie}
 Ji\v{r}\'i Wiedermann. Characterizing the super-Turing computing power and efficiency of classical fuzzy Turing machines, \emph{Theoretical Computer Science} 317(1--3):61--69, 2004.
\end{thebibliography}
\end{document}